\documentclass{article}
\usepackage{amsmath}
\usepackage{amsthm}
\usepackage{amssymb}
\usepackage{graphicx}
\usepackage{graphics}
\usepackage{listings}
\usepackage{multirow}
\usepackage{framed}
\usepackage{hyperref}

\usepackage[algo2e,ruled,vlined,]{algorithm2e}
\usepackage{algorithmic}

\usepackage{algorithm}
\usepackage{algorithmic}

\usepackage{tikz}
\usepackage{pgf}
\usetikzlibrary{arrows}
\usepackage{mathtools}
\usepackage{mathrsfs}

\usepackage{color}
\definecolor{red}{rgb}{1,0,0}

\newcommand{\abs}[1]{\left\vert#1\right\vert}

\newcommand{\ms}{\medskip}
\newcommand{\bpf}{\begin{proof}}
\newcommand{\epf}{\end{proof}\ms}

\newtheorem{theorem}{Theorem}
\newtheorem{corollary}[theorem]{Corollary}
\newtheorem{lemma}[theorem]{Lemma}
\newtheorem{proposition}[theorem]{Proposition}

\newtheorem{claim}{Claim}

\theoremstyle{definition}
\newtheorem{definition}{Definition}
\newtheorem{question}{Question}
\newtheorem{conjecture}{Conjecture}

\begin{document}

\title{The zero forcing polynomial of a graph}

\author{Kirk Boyer\thanks{Department of Mathematics, University of Denver, Denver, CO, 80208, USA (kirk.boyer@du.edu)} \and 
Boris Brimkov\thanks{Department of Computational and Applied Mathematics, Rice University, Houston, TX, 77005, USA (boris.brimkov@rice.edu)} \and 
Sean English\thanks{Department of Mathematics, Western Michigan University, Kalamazoo, MI, 49008, USA (sean.j.english@wmich.edu)} \and 
Daniela Ferrero\thanks{Department of Mathematics, Texas State University, San Marcos, TX, 78666, USA (df20@txstate.edu)} \and 
Ariel Keller\thanks{Department of Mathematics and Computer Science, Emory University,  Atlanta, GA, 30322, USA (ariel.keller@emory.edu)} \and 
Rachel Kirsch\thanks{Department of Mathematics, University of Nebraska, Lincoln, NE, 68588, USA (rkirsch@huskers.unl.edu)} \and 
Michael Phillips\thanks{Department of Mathematical and Statistical Science, University of Colorado Denver, Denver, CO, 80204, USA (michael.2.phillips@ucdenver.edu)} \and  
Carolyn Reinhart\thanks{Department Mathematics, Iowa State University, Ames, IA, 50011, USA (reinh196@iastate.edu)}}

\date{}

\maketitle

\begin{abstract}
Zero forcing is an iterative graph coloring process, where given a set of initially colored vertices, a colored vertex with a single uncolored neighbor causes that neighbor to become colored. A zero forcing set is a set of initially colored vertices which causes the entire graph to eventually become colored. In this paper, we study the counting problem associated with zero forcing. We introduce the zero forcing polynomial of a graph $G$ of order $n$ as the polynomial $\mathcal{Z}(G;x)=\sum_{i=1}^n z(G;i) x^i$, where $z(G;i)$ is the number of zero forcing sets of $G$ of size $i$. We characterize the extremal coefficients of $\mathcal{Z}(G;x)$, derive closed form expressions for the zero forcing polynomials of several families of graphs, and explore various structural properties of $\mathcal{Z}(G;x)$, including multiplicativity, unimodality, and uniqueness. 

\smallskip

\end{abstract}

\section{Introduction}

Given a graph $G=(V,E)$ and a set $S \subseteq V$ of initially colored vertices, the \emph{zero forcing color change rule} dictates that at each timestep, a colored vertex $u$ with a single uncolored neighbor $v$ \emph{forces} that neighbor to become colored. 
The \emph{closure} of $S$ is the set of colored vertices obtained after the color change rule is applied until no new vertex can be forced. A \emph{zero forcing set} is a set whose closure is all of $V$; the \emph{zero forcing number} of $G$, denoted $Z(G)$, is the minimum cardinality of a zero forcing set. Zero forcing and similar processes have been independently studied in the contexts of linear algebra \cite{AIM-Workshop}, quantum physics \cite{quantum1}, theoretical computer science \cite{fast_mixed_search}, and power network monitoring \cite{powerdom3,powerdom2}. Zero forcing has applications in modeling various physical phenomena \cite{zf_physics,logic1,kforcing3,disease} and in bounding or approximating other graph parameters \cite{AIM-Workshop,zf_tw,zf_physics,butler}. See also \cite{Barioli,brimkov2,positive_zf2,signed_zf,fractional_zf} for variants of zero forcing, which are typically obtained by modifying the zero forcing color change rule or adding certain restrictions to the structure of a forcing set.

In this paper, we study the counting problem associated with zero forcing, i.e., characterizing and counting the distinct zero forcing sets of a graph. The set of minimum zero forcing sets of a graph has been alluded to previously in the context of \emph{propagation time} \cite{proptime_oriented,proptime1,proptime2}, where the objective is to find the largest or smallest number of timesteps it takes for the graph to be colored by a minimum zero forcing set. Similarly, the set of all zero forcing sets of a graph has been used in the context of \emph{throttling} \cite{throttling,throttling2}, where the objective is to minimize the sum of the size of a zero forcing set and the number of timesteps it takes for that zero forcing set to color the graph. In order to study the collection of zero forcing sets of a graph in a more general framework, we introduce the \emph{zero forcing polynomial} of a graph, which counts the number of distinct zero forcing sets of a given size. 
\begin{definition}
Let $G$ be a graph on $n$ vertices and $z(G;i)$ be the number of zero forcing sets of $G$ with cardinality $i$. The \emph{zero forcing polynomial} of $G$ is defined as
\begin{equation*}
\mathcal{Z}(G;x)=\sum_{i=1}^n z(G;i) x^i.
\end{equation*}
\end{definition}

In this paper, we study the basic algebraic and graph theoretic properties of the zero forcing polynomial, present structural and extremal results about its coefficients, and give closed form expressions for the zero forcing polynomials of several families of graphs\footnote{Some of these results, as well as analogous results about a connected variant of zero forcing, have been reported in \cite{brimkov_thesis}.}. As an application, we relate the zero forcing polynomial to special sets of vertices used in integer programming approaches for computing the zero forcing number. In general, graph polynomials contain important information about the structure and properties of graphs that can be extracted by algebraic methods. In particular, the values of graph polynomials at specific points, as well as their coefficients, roots, and derivatives, often have meaningful interpretations. Such information and other unexpected connections between graph theory and algebra are sometimes discovered long after a graph polynomial is originally introduced (see, e.g., \cite{brylawski,stanley}).

The study of graph polynomials was motivated by the Four Color Conjecture, when Birkhoff and Whitney \cite{birkhoff,whitney} introduced the \emph{chromatic polynomial} which counts the number of proper colorings of a graph $G$. Tutte generalized the chromatic polynomial into the two-variable \emph{Tutte polynomial} \cite{tutte2}, which also contains as special cases the \emph{flow polynomial}, \emph{reliability polynomial}, \emph{shelling polynomial}, and \emph{Jones polynomial}; see \cite{reliability_poly,tuttemain,jones} for more details. Graph polynomials which are not direct specializations of the Tutte polynomial have also been studied. For example, the \emph{domination polynomial} \cite{dom_poly1} of a graph $G$ counts the number of dominating sets of $G$. Work in this direction includes derivations of recurrence relations \cite{dom_poly3}, analysis of the roots  \cite{dom_poly2}, and characterizations for specific graphs \cite{dom_poly4,dom_poly5}. Similar results have been obtained for the \emph{connected domination polynomial} \cite{connected_dom_poly}, \emph{independence polynomial} \cite{indep_poly}, \emph{clique polynomial} \cite{clique_poly}, \emph{vertex cover polynomial} \cite{vertex_cover_poly}, and \emph{edge cover polynomial} \cite{edge_cover_poly}, which are defined as the generating functions of their eponymous sets. For more definitions, results, and applications of graph polynomials, see the survey of Ellis-Monaghan and Merino \cite{tuttemain2} and the bibliography therein. 

This paper is organized as follows. In the next section, we recall some graph theoretic notions and notation. In Section \ref{section_extremal}, we  characterize the extremal coefficients of the zero forcing polynomial. In Section \ref{section_characterizations}, we give closed form expressions for the zero forcing polynomials of several families of graphs. In Section \ref{section_structural}, we explore various structural properties of the zero forcing polynomials of general graphs. We conclude with some final remarks and open questions in Section \ref{section_conclusion}.

\section{Preliminaries}

A graph $G=(V,E)$ consists of a vertex set $V$ and an edge set $E$ of two-element subsets of $V$. The \emph{order} of $G$ is denoted by $n=|V|$. Two vertices $v,w\in V$ are \emph{adjacent}, or \emph{neighbors}, if $\{v,w\}\in E$; we will write $v\sim w$ if $v$ and $w$ are adjacent. The \emph{neighborhood} of $v\in V$ is the set of all vertices which are adjacent to $v$, denoted $N(v;G)$; the \emph{degree} of $v\in V$ is defined as $d(v;G)=|N(v;G)|$. The \emph{closed neighborhood} of $v$ is the set $N[v;G]=N(v;G)\cup \{v\}$. The minimum degree and maximum degree of $G$ are denoted by $\delta(G)$ and $\Delta(G)$, respectively. The dependence of these parameters on $G$ can be omitted when it is clear from the context. Given $S \subseteq V$, the \emph{induced subgraph} $G[S]$ is the subgraph of $G$ whose vertex set is $S$ and whose edge set consists of all edges of $G$ which have both endpoints in $S$. An isomorphism between graphs $G_1$ and $G_2$ will be denoted by $G_1\simeq G_2$. A \emph{leaf}, or \emph{pendent}, is a vertex with degree 1; an \emph{isolated vertex} is a vertex of degree 0. A \emph{dominating vertex} is a vertex which is adjacent to all other vertices. A \emph{cut vertex} is a vertex which, when removed, increases the number of connected components in $G$. A \emph{biconnected component} of $G$ is a maximal subgraph of $G$ which has no cut vertices. The path, cycle, complete graph, and empty graph on $n$ vertices will respectively be denoted $P_n$, $C_n$, $K_n$, $\overline{K}_n$. The complete multipartite graph whose parts have sizes $a_1,\ldots,a_k$ will be denoted $K_{a_1,\ldots,a_k}$. 

\sloppypar Given two graphs $G_1$ and $G_2$, the \emph{disjoint union} $G_1\dot\cup G_2$ is the graph with vertex set $V(G_1)\dot\cup V(G_2)$ and edge set $E(G_1)\dot\cup E(G_2)$. The \emph{join} of $G_1$ and $G_2$, denoted $G_1\lor G_2$, is the graph obtained from $G_1\dot\cup G_2$ by adding an edge between each vertex of $G_1$ and each vertex of $G_2$. The \emph{Cartesian product} of $G_1$ and $G_2$, denoted $G_1\square G_2$, is the graph with vertex set $V(G_1)\times V(G_2)$, where vertices $(u,u')$ and $(v,v')$ are adjacent in $G_1\square G_2$ if and only if either $u = v$ and $u'$ is adjacent to $v'$ in $G_2$, or $u' = v'$ and $u$ is adjacent to $v$ in $G_1$. For other graph theoretic terminology and definitions, we refer the reader to \cite{bondy}.

A \emph{chronological list of forces} of a set $S\subseteq V(G)$ is a sequence of forcing steps applied to obtain the closure of $S$ in the order they are applied; note that there can be initially colored vertices which do not force any vertex. A \emph{forcing chain} for a chronological list of forces is a maximal sequence of vertices $(v_1,\ldots,v_k)$ such that $v_i$ forces $v_{i+1}$ for $1\leq i\leq k-1$. Each forcing chain corresponds to a distinct path in $G$, one of whose endpoints is an initially colored vertex and the rest of which is uncolored at the initial timestep; we will say the initially colored vertex \emph{initiates} the forcing chain, and we will call the other endpoint of the forcing chain a \emph{terminal} vertex. 

Given a set $X$, we denote by ${X\choose k}$ the set of all $k$-element subsets of $X$; we denote by $2^X$ the power set of $X$. Given integers $a$ and $b$ with $0\leq a<b$, we adopt the convention that ${a\choose b}=0$. For any positive integer $n$, $[n]$ denotes the set $\{1,\ldots,n\}$.

\section{Extremal coefficients of $\mathcal{Z}(G;x)$}
\label{section_extremal}
In this section, we characterize some of the extremal coefficients of $\mathcal{Z}(G;x)$. 

\begin{theorem}
\label{thm_zf_polynomial_properties}
Let $G=(V,E)$ be a graph. Then,
\begin{enumerate}
\item $z(G;n)=1$,
\item $z(G;n-1)=|\{v\in V:d(v)\neq 0\}|$,
\item $z(G;n-2)=|\{\{u,v\}\subset V:u\neq v,d(u)\neq 0, d(v)\neq 0, N(u)\backslash\{v\}\neq N(v)\backslash\{u\}\}|$,
\item $z(G;1)=\begin{cases}
2 &\text{ if } G\simeq P_n, n\geq 2,\\
1 &\text{ if } G\simeq P_1,\\
0 &\text{ otherwise}.
\end{cases}$
\end{enumerate}
\end{theorem}

\proof
The numbers of the proofs below correspond to the numbers in the statement of the theorem. 
\begin{enumerate}
\item[1.] $V$ is the only zero forcing set of size $n$. 
\item[2.] Any non-isolated vertex $v$ has a neighbor which can force $v$. Thus, each set which excludes one non-isolated vertex of $G$ is a zero forcing set of size $n-1$; conversely, no set which excludes an isolated vertex is a zero forcing set. 
\item[3.] Let $u,v$ be two non-isolated vertices of $G$. If $N(u)\backslash\{v\}\neq N(v)\backslash\{u\}$, then there is a vertex $w$ adjacent to one of $u$ and $v$, but not the other. Suppose $u\sim w$; then, $V\backslash\{u,v\}$ is a zero forcing set since $w$ can force $u$ and any neighbor of $v$ can force $v$. On the other hand, a pair of vertices $u,v$ which does not satisfy these conditions cannot be excluded from a zero forcing set, since every vertex which is adjacent to one will be adjacent to the other, and hence no vertex will be able to force $u$ or $v$. 
\item[4.] The only graph with zero forcing number 1 is $P_n$. Thus, if $G\not\simeq P_n$, $z(G;1)=0$. If $G\simeq P_n$ and $n\geq 2$, either end vertex of the path is a zero forcing set. If $n=1$, there is a single zero forcing set.
\qed
\end{enumerate}

We next consider the coefficient $z(G;Z(G))$, i.e., the number of minimum zero forcing sets of a graph $G$. Reversing the forcing chains associated with a zero forcing set produces another zero forcing set; thus, a (non-empty) graph cannot have a unique minimum zero forcing set. Moreover, a trivial upper bound on the coefficient $z(G;Z(G))$ is $z(G;Z(G))\leq {n\choose Z(G)}$. We now classify the families of graphs for which this bound holds with equality.

\begin{theorem}
If $G$ is a graph on $n$ vertices with $z(G;Z(G))=\binom{n}{Z(G)}$, then $G\simeq K_n$ or $G\simeq \overline{K_n}$.
	\end{theorem}

\begin{proof}
If $z(G;Z(G))=\binom{n}{Z(G)}$, then every set of vertices of size $Z(G)$ is a zero forcing set. Clearly, when $G\simeq K_n$ or $G\simeq \overline{K_n}$, this property holds. We will assume that $G\not\simeq \overline{K_n}$ and show that $G\simeq K_n$. If $Z(G)=1$, it is easy to see that $G\simeq K_2$; thus, assume henceforth that $Z(G)\geq 2$. Since $G\not\simeq \overline{K_n}$, it follows that $Z(G)<n$. Then, if $G$ contains an isolated vertex $v$ (and $G\not\simeq K_1$), a set of size $Z(G)$ which omits $v$ is not forcing, so $G$ could not have the desired property. 
Thus, $G$ does not have isolated vertices. 

We next claim that $\delta(G)=Z(G)$. Indeed, for any graph $G$, $\delta(G)\leq Z(G)$. Suppose there exists a vertex $v$ with $d(v)<Z(G)$, and let $w$ be a neighbor of $v$ (which exists since $v$ is not an isolated vertex). Then, let $S$ be a set consisting of $N[v]\setminus\{w\}$ together with $Z(G)-d(v)-1$ vertices not adjacent to $v$. Note that this set contains $Z(G)-1$ elements so it is not a zero forcing set, but this set forces $w$. Once this initial force is performed, we have a set with $Z(G)$ colored vertices, which must be a forcing set by our assumptions -- a contradiction. Thus $\delta(G)= Z(G)$.

Let $v$ be any vertex, and let $S\subseteq N(v)$ be a set of $Z(G)$ vertices. Then, $S$ is a zero forcing set. Let $u\in S$ be a vertex which can perform the first force; then, $d(u)=\delta(G)=Z(G)$. Since $N(u)$ is a set of size $Z(G)$, it is also a zero forcing set. Let $w\in N(u)$ be a vertex which can perform the first force; then, $d(w)=Z(G)$ and $N(w)=N(u)\backslash\{w\}\cup\{u\}$. Note that $N(u)\backslash \{w\}=N(w)\backslash\{u\}$, so any set which excludes at least one of $u$ and $w$ cannot be a zero forcing set. This implies that $Z(G)=n-1$, since if $Z(G)\leq n-2$, a set of size $Z(G)$ which excludes both $u$ and $w$ would not be zero forcing. Since $K_n$ is the only graph without isolates which has zero forcing number $n-1$, it follows that $G\simeq K_n$.
\end{proof}

\section{Characterizations of $\mathcal{Z}(G;x)$ for specific graphs}
\label{section_characterizations}

\noindent In this section, we give closed form expressions for the zero forcing polynomials of certain families of graphs.

\begin{proposition}
\label{prop_clique}
For $n\geq 2$, $\mathcal{Z}(K_n;x)=x^n+n x^{n-1}$.
\end{proposition}
\begin{proof}
$Z(K_n)=n-1$, so $z(K_n;i)=0$ for $i<n-1$. By Theorem \ref{thm_zf_polynomial_properties}, $z(K_n;n-1)=n$ and $z(K_n;n)=1$, so $\mathcal{Z}(K_n;x)=x^n+n x^{n-1}$.
\end{proof}

\begin{proposition}
\label{prop_multipartite}
If $a_1,\ldots,a_k\geq 2$, $\mathcal{Z}(K_{a_1,\ldots,a_k};x)=(\sum_{1\leq i<j\leq k}a_ia_j)x^{n-2}+n x^{n-1}+x^n$.
\end{proposition}
\proof
$Z(K_{a_1,\ldots,a_k})=n-2$, so $z(K_{a_1,\ldots,a_k};i)=0$ for $i<n-2$. Each minimum zero forcing set of $K_{a_1,\ldots,a_k}$ excludes a vertex from two of the parts of $K_{a_1,\ldots,a_k}$; there are $\sum_{1\leq i<j\leq k}a_ia_j$ ways to pick such a pair of vertices, so $z(K_{a_1,\ldots,a_k};n-2)=\sum_{1\leq i<j\leq k}a_ia_j$. By Theorem \ref{thm_zf_polynomial_properties}, $z(K_{a_1,\ldots,a_k};n-1)=n$ and $z(K_{a_1,\ldots,a_k};n)=1$, so $\mathcal{Z}(K_{a_1,\ldots,a_k};x)=x^n+n x^{n-1}+(\sum_{1\leq i<j\leq k}a_ia_j)x^{n-2}$.
\qed

\begin{proposition}
\label{prop_path}
For $n\geq 1$, $\mathcal{Z}(P_n;x)=\sum_{i=1}^n({n\choose i}-{n-i-1\choose i})x^i$.

\end{proposition}
\proof
The sets of $P_n$ of size $i$ which are not forcing are those which do not contain an end-vertex of the path and do not contain adjacent vertices. To count the number of non-forcing sets of size $i$, we can use the following argument: there are $n-i$ indistinguishable uncolored vertices to be placed in the $i+1$ positions around the colored vertices, where each position must receive at least one uncolored vertex (in order for there not to be any adjacent colored vertices and for the end-vertices to be uncolored). There are ${(n-i)-1\choose (i+1)-1}$ ways to choose the positions of the uncolored vertices. Thus, there are ${n\choose i}-{n-i-1\choose i}$ zero forcing sets of size $i$. Note that when $i\geq \lceil \frac{n}{2}\rceil$, then $n-i-1<i$, and hence ${n-i-1\choose i}=0$. Thus, $\mathcal{Z}(P_n;x)=\sum_{i=1}^n({n\choose i}-{n-i-1\choose i})x^i$.
\qed

\begin{proposition}
\label{prop_cycle}
For $n\geq 3$, $\mathcal{Z}(C_n;x)=\sum_{i=2}^n({n\choose i}-\frac{n}{i}{n-i-1\choose i-1})x^i$.
\end{proposition}
\proof
The sets of $C_n$ of size $i$ which are not forcing are those which do not contain adjacent vertices. To count the number of non-forcing sets of size $i$, we can use the following argument: first, select a representative vertex $v$ of $C_n$ and color it; this can be done in $n$ ways. There are $n-i$ indistinguishable uncolored vertices to be placed in the $i$ positions around the remaining $i-1$ colored vertices, where each position must receive at least one uncolored vertex (in order for there not to be any adjacent colored vertices). There are ${n-i-1\choose i-1}$ ways to choose the positions of the uncolored vertices. This can be done for each of the $n$ choices of a representative vertex $v$; however, since we are interested in sets without a representative vertex, and since each set has been counted $i$ times with a different representative vertex, we must divide this quantity by $i$. Thus, there are $\frac{n}{i}{n-i-1\choose i-1}$ ways to choose $i$ vertices from $C_n$ so that no two are adjacent. It follows that there are ${n\choose i}-\frac{n}{i}{n-i-1\choose i-1}$ zero forcing sets of size $i$. Note that when $i\geq \lfloor \frac{n}{2}\rfloor+1$, then $n-i-1<i-1$, and hence ${n-i-1\choose i-1}=0$. Thus, $\mathcal{Z}(C_n;x)=\sum_{i=2}^n({n\choose i}-\frac{n}{i}{n-i-1\choose i-1})x^i$.
\qed

\medskip

\subsection{Wheels}

Let $S$ be a set of $k$ vertices of a cycle $C_n$; we will say these $k$ vertices are \emph{consecutive} if $C_n[S]$ is a path.

\begin{lemma}
\label{lemma_wheel}
Given integers $n>k\geq m\geq 3$, the number of ways to select $k$ labeled vertices of $C_n$ such that at least $m$ of the selected vertices are consecutive is
$$
R_m(n,k):=\sum_{t=1}^n (-1)^{t-1}\frac{n}{t}\binom{n-mt-1}{t-1}\binom{n-(m+1)t}{k-mt}.
$$
\end{lemma}

\proof
Label the vertices of $C_n$ with $[n]$ so that for $1\leq i\leq n-1$, the vertices with labels $i$ and $i+1$ are adjacent; define 
\begin{equation*}
A_i := \left\{S \in \binom{[n]}{k}: i,\dots,i+m-1 \in S, i-1 \not\in S\right\},
\end{equation*}
where all arithmetic is mod $n$. Then $\bigcup_{i\in [n]} A_i$ is the set of all ways to pick $k$ labeled vertices from $C_n$ so that at least $m$ of them are consecutive, or the number of $k$-sets $S \subseteq V(C_n)$ so that $C_n[S]$ contains $P_m$ as a subgraph. By the principle of inclusion-exclusion, we have 
\begin{equation}\label{pie}
\abs{\bigcup_{i\in [n]} A_i}
	= \sum_{\emptyset \neq J \subseteq [n]}(-1)^{|J|-1} \abs{\bigcap_{j \in J} A_j}. \\
\end{equation}
We will now give a closed form expression for the right-hand side of (\ref{pie}). First, we will determine the number of sets $J$ of size $t$ for which $\cap_{j\in J}A_j$ is nonempty as well as the size of the intersection in terms of $t$. To this end, we make an observation about the spacing around $C_n$ of the elements of $J$ when the intersection is nonempty. Pick some distinct $i, j \in J$, $S \in \cap_{j\in J} A_j$, and $j - i \mod n < i - j \mod n$ (i.e. the shorter distance around the cycle is from $i$ to $j$). Then $i+1, \dots, i+m-1 \in S$ and $j-1 \notin S$, so at least these $m$ vertices lie strictly between $i$ and $j$.

Each $j \in J$ determines the inclusion in $S$ of $m+1$ consecutive vertices, $m$ of which are in $S$, and for distinct $i, j \in J$, these sets of consecutive vertices are disjoint. Therefore $S \in \cap_{j\in J}A_j$ determines the inclusion in $S$ of $(m+1)|J|$ vertices, $m|J|$ of which are in $S$. We distinguish between the $S \in \cap A_j$ by choosing which $k-m|J|$ of the remaining $n-(m+1)|J|$ vertices are in $S$. With $t = |J|$,
\begin{equation*}
\abs{\bigcap_{j \in J} A_j} = \binom{n-(m+1)t}{k-mt}.
\end{equation*}
All that remains is to determine how many subsets $J$ of size $t$ have nonempty $\bigcap_{j \in J} A_j$. First, assume that $J$ contains the vertex with label 1. Note that $J$ is non-empty if and only if there are at least $m$ vertices between each vertex in $J$. This is equivalent to the well-known problem of placing $n-t$ indistinguishable balls into $t$ distinguishable boxes with each box containing at least $m$ balls, which can be done in $\binom{n-mt-1}{t-1}$ ways.


Fixing any initial vertex of $C_n$ would result in the same number of sequences, but $n\binom{n-mt-1}{t-1}$ overcounts each subset $J$ by a factor of $t$, once for each of $t$ distinct choices from $J$ for this initial vertex. Therefore, the number of subsets $J$ of size $t$ where $\bigcap_{j \in J} A_j$ is nonempty is $\frac{n}{t}\binom{n-mt-1}{t-1}$. Hence \eqref{pie} simplifies to
\begin{equation*}
\abs{\bigcup_{i\in V(C_n)} A_i}= \sum_{t=1}^n (-1)^{t-1}\frac{n}{t}\binom{n-mt-1}{t-1}\binom{n-(m+1)t}{k-mt}.\tag*{\qed}
\end{equation*}
\vspace{9pt}

The \emph{wheel} on $n$ vertices, denoted $W_n$, is the graph obtained by adding a dominating vertex to $C_{n-1}$.

\begin{theorem}
For $n\geq 4$,
$\mathcal{Z}(W_n;x) = \sum_{i=1}^n (z(C_{n-1};i-1)+R_3(n-1,i))x^i$.
\end{theorem}
\proof
Let $v$ be the dominating vertex of $W_n$. The zero forcing sets of $W_n$ of size $i$ can be partitioned into those which contain $v$ and those which do not contain $v$. Since $v$ is a dominating vertex of $W_n$, $S$ is a zero forcing set of $W_n$ of size $i$ which contains $v$ if and only if $S\backslash \{v\}$ is a zero forcing set (of size $i-1$) of $W_n-v$. Since $W_n-v\simeq C_{n-1}$, by Proposition \ref{prop_cycle}, the number of zero forcing sets of $W_n$ of size $i$ which contain $v$ is $z(C_{n-1};i-1)$.
Next, to count the number of zero forcing sets of $W_n$ of size $i$ which do not contain $v$, note that $S \subseteq V(G)\backslash\{v\}$ is a zero forcing set of $W_n$ if and only if $S$ induces a path on at least three vertices, since the dominating vertex must be forced before any other vertices can force. By Lemma \ref{lemma_wheel}, there are $R_3(n-1,i)$ such sets (for $i<3$, we can define $R_3(n-1,i)=0$). By adding $z(C_{n-1};i-1)$ and $R_3(n-1,i)$, we conclude that $\mathcal{Z}(W_n;x)$ is as desired.
\qed


\subsection{Threshold graphs}

A graph $G=(V,E)$ is a \emph{threshold graph} if there exists a real number $T$ and vertex weight function $w:V\to \mathbb{R}$ such that $uv\in E$ if and only if $w(u)+w(v)\geq T$. For a binary string $B$, the \emph{threshold graph generated by $B$}, written $T(B)$, is the graph whose vertices are the symbols in $B$, and which has an edge between a pair of symbols $x$ and $y$ with $x$ to the left of $y$ if and only if $y=1$. It was shown by Chv{\'a}tal and Hammer \cite{chvatal} that every threshold graph is generated by some binary string, and two distinct binary strings yield the same threshold graph if and only if they only differ in the first symbol of the string. To avoid this repetition, without loss of generality, we will deal exclusively with binary strings whose first and second symbols are the same.

A \emph{block} of a binary string $B$ is a maximal contiguous substring consisting either of only 0s or only 1s. A partition of $B$ into its blocks is called the \emph{block partition} of $B$, and when the block partition of $B$ has $t$ blocks, we label the blocks $B_{i}$, $1\leq i\leq t$, and write $B = B_{1}B_{2}\dots B_{t}$. Each $B_{i}$ consisting of 0s is called a 0-block, and each $B_{i}$ consisting of 1s is a 1-block. Similarly, we will call a vertex in a 0-block a 0-vertex, and a vertex in a 1-block a 1-vertex. Finally, we will assume that $T(B)$ is connected, i.e. that $B_{t}$ is a 1-block.

\begin{lemma}
\label{lemma: neighbors of 1s}
Let $T = T(B)$ be a threshold graph on binary string $B$. Then,
\begin{enumerate}
\item[1)] If any 1-vertex is uncolored, then no 1-vertex can force a 0-vertex.
\item[2)] If more than one 1-vertex is uncolored, then no 1-vertex can force any of its neighbors.
\end{enumerate}
\end{lemma}

\proof
All 1-vertices are adjacent. In the first case, any colored 1-vertex with an uncolored 0-neighbor has the uncolored 1-neighbor as well, and thus cannot force either neighbor. In the second case, every colored 1-vertex has at least two uncolored 1-neighbors and thus cannot force either of them.
\qed

\begin{theorem}
\label{theorem: zfs characterization}
Let $T=T(B)$ be a threshold graph on binary string $B$, $|B|\geq 2$. A set $S \subseteq V(T)$ is a zero forcing set if and only if
\begin{enumerate}
\item[1)] $S$ excludes at most one vertex from each block of $B$, and
\item[2)] In $B$, between any two 1-vertices not in $S$, there is a 0-vertex in $S$.
\end{enumerate}
\end{theorem}

\begin{proof}
($\Rightarrow$) Let $S$ be a zero forcing set of $T$. We will show conditions 1) and 2) both hold. Suppose first that $x$ and $y$ are vertices in the same block. Then $N(x)\backslash\{y\}= N(y)\backslash\{x\}$, so if $S$ excludes both $x$ and $y$, then any vertex that could force $x$ or $y$ would have at least two uncolored neighbors, neither of which can be forced. Therefore, $S$ cannot exclude both $x$ and $y$ and so condition 1) holds.

Now, suppose there are two 1-vertices $x$ and $y$ not in $S$ and that every 0-vertex in $B$ between $x$ and $y$ is also not in $S$. By Lemma \ref{lemma: neighbors of 1s}, no 1-vertex can force anything until either $x$ or $y$ is forced by some 0-vertex. Since 0-vertices are not adjacent to other 0-vertices, they cannot force other 0-vertices; thus, only 0-vertices in $S$ could possibly force $x$ or $y$. However, 0-vertices that come before $x$ and $y$ in $B$ are adjacent to both $x$ and $y$ (and hence cannot force), 0-vertices between $x$ and $y$ are not in $S$, and 0-vertices that come after $x$ and $y$ are not adjacent to either. Therefore $S$ must contain a 0-vertex between any two 1-vertices not in $S$ and condition 2) holds.

($\Leftarrow$) Let $S$ be a set which satisfies conditions 1) and 2). We will show that $S$ is a zero forcing set. If $B$ has a single block, this block is a 1-block and $T(B)$ is a complete graph, and both conditions clearly hold. Similarly, if $B$ has a single 1-block which consists of a single element, then $T(B)$ is a star and both conditions hold. Hence, we can assume $B$ has at least two 1-vertices.

We will show that the forcing in $T(B)$ happens sequentially in two stages. First, the 1-vertices not in $S$ are forced in order from right to left, each by a colored 0-vertex in the 0-block immediately to its left. Second, the 0-vertices not in $S$ are forced in order from left to right, each by a colored 1-vertex in the 1-block immediately to its right. These stages are described in more detail below. Recall our assumptions that $T(B)$ is connected and that the first element in $B$ is the same as the second; hence, the rightmost block $B_t$ is a 1-block and the leftmost block $B_1$ has at least two vertices.

\emph{Stage 1:} Let $v$ be the rightmost uncolored 1-vertex in $B$. If $v$ is the only uncolored 1-vertex in $B$, then since $B_1$ has at least two vertices, by condition 1), one of them must be in $S$ and can therefore force $v$. If $v$ is not the only uncolored 1-vertex in $B$, by condition 2), there is a colored 0-vertex $w$ between $v$ and the next uncolored 1-vertex. The  neighborhood of $w$ is the set of 1-vertices to the right of $w$, which is exactly $v$ and all of the 1-vertices to the right of $v$, which are already colored. Therefore $w$ can force $v$. Inductively, all 1-vertices in $B$ get colored.

\emph{Stage 2:} Let $v$ be the leftmost uncolored 0-vertex in $B$. Consider a 1-vertex $w$ in the block immediately to the right of $v$. The neighborhood of $w$ is the set of all vertices to the left of $w$, and the set of 1-vertices to the right of $w$. Since Stage 1 is complete, $w$ and all other 1-vertices are colored.  By condition 1), $v$ is the only uncolored 0-vertex in its block; moreover, all 0-vertices in the blocks to the left of $v$ are already colored. Therefore $w$ can force $v$. Inductively, all 0-vertices in $B$ get colored.
\end{proof}

This characterization of zero forcing sets can also be stated in terms of a selection of elements within blocks of the binary string, as follows.

\begin{corollary}
\label{corollary: block description of zfs}
Let $B_1\ldots B_t$ be the block partition for a binary string $B$. Then for any zero forcing set $S$ of the threshold graph $T(B)$, there is a set of block indices $A \subseteq \{1,\ldots,t\}$ and a corresponding set of symbol indices $J = \{j_{i}\in \{1,\ldots,|B_{i}|\}:i \in A\}$ that uniquely identifies $S$. In particular, $S$ is the set of size $n-|A|$ that contains all vertices except the $j_i^{\text{th}}$ vertex from block $B_i$ for each $i\in A$.
\end{corollary}

Corollary \ref{corollary: block description of zfs} shows that any zero forcing set in a threshold graph can be described by listing the blocks in which a single vertex is not part of the set, and identifying which vertex this is within each block. However, in order to satisfy Condition 2) of Theorem \ref{theorem: zfs characterization}, the selection of blocks must be made carefully when a vertex in a $0$-block of size $1$ is excluded from a zero forcing set; see Algorithm 1 for a formal description of how to find all sets of block indices corresponding to zero forcing sets.

\begin{algorithm2e}[h]
\textbf{Input:} Binary string $B=B_1\ldots B_t$\;
\textbf{Output:} Set $\mathcal{A}$ of all sets of block indices corresponding to zero forcing sets of $T(B)$\;
\uIf{$t \mod 2=1$}{
$\mathcal{A}_1\leftarrow\{(\emptyset,0),(\{1\},1)\}$\;
$s\leftarrow 3$\;
}
\Else{
$\mathcal{A}_2\leftarrow\{(\emptyset,0),(\{1\},0),(\{2\},1),(\{1,2\},1)\}$\;
$s\leftarrow 4$\;
}
\While{$s\leq t$}{
$\mathcal{A}_s\leftarrow\emptyset$\;
\For{$(A,k)\in\mathcal{A}_{s-2}$}{
\uIf{$k=1$ \textbf{\emph{and}} $|B_{s-1}|=1$}{
$\mathcal{A}_s\leftarrow \mathcal{A}_s \cup \{(A,0),(A\cup\{s-1\},1),(A\cup\{s\},1)\}$\;}
\Else{
$\mathcal{A}_s\leftarrow \mathcal{A}_s \cup \{(A,0),(A\cup\{s-1\},0),(A\cup\{s\},1),(A\cup\{s-1,s\},1)\}$\;}
}
$s\leftarrow s+2$\;
}
\Return{$\mathcal{A}\leftarrow\{A:(A,k)\in \mathcal{A}_t, k\in \{0,1\}\}$\;}
\caption{Finding all sets of block indices corresponding to zero forcing sets}
\end{algorithm2e}

\begin{theorem}
\label{threshold zero forcing polynomial}
Let $B$ be a binary string with block partition $B_1\ldots B_t$, and let $\mathcal{A}$ be the output of Algorithm 1. Then,

\begin{equation}\label{equation: threshold polynomial theorem}
\mathcal{Z}(T(B); x) = \sum_{A \in \mathcal{A}} \left[ \left( \prod_{i\in A} |B_i| \right) x^{n - |A|} \right]= \sum_{k=n-t}^n\left[\sum_{\substack{A\in \mathcal{A} \\ |A| = n-k}}\left(\prod_{i\in A} |B_i| \right)\right]x^{k}.
\end{equation} 
\end{theorem}

\proof
We will first show that $\mathcal{A}$ is the collection of all sets of block indices for which there is at least one zero forcing set $S$ fitting the description of Corollary \ref{corollary: block description of zfs}. More precisely, we will show that $A\in \mathcal{A}$ if and only if for each set $J=\{j_i:i \in A\}$ with $j_i\in B_i$ for $i\in A$, the set $S:=V(T(B))\backslash J$ is a zero forcing set of $T(B)$.

In Algorithm 1, for $s\leq t$ such that $B_s$ is a 1-block, each element of $\mathcal{A}_s$ is an ordered pair $(A,k)$: the first entry is a set $A$ of block indices such that there exists a zero forcing set of the threshold graph generated by $B_1\dots B_s$ that excludes one vertex from each block whose index is in $A$; the second entry $k$ indicates when vertices from a $1$-block cannot be excluded. In particular, $k=0$ indicates that either no $1$-block has been excluded thus far or a $1$-block has been excluded and there is a $0$-vertex included after the rightmost excluded $1$-block, and $k=1$ indicates that a $1$-block has been excluded and there is no $0$-vertex included after the rightmost excluded $1$-block.

Let $A$ be an arbitrary element of $\mathcal{A}$ and $J$ be an arbitrary set of indices, one from each block $B_i$ for $i\in A$. We will show that $S:=V(T(B))\backslash J$ is a zero forcing set of $T(B)$. $S$ satisfies Condition 1) of Theorem \ref{theorem: zfs characterization} since $S$ excludes only one vertex from each block $B_i$, $i\in A$. To see that $S$ satisfies Condition 2), first note that in order for a $1$-block to be excluded at any stage of Algorithm 1, $(A\cup \{s\},1)$ or $(A\cup \{s-1,s\},1)$ is added to $\mathcal{A}_s$. In both cases, the indicator entry of the corresponding set is $k=1$. Now, let $a,b\in A$ be the indices of two $1$-blocks selected, and assume without loss of generality that there are no other $1$-block indices in $A$ between $a$ and $b$. If any $0$-block between $B_a$ and $B_b$ has length greater than $1$, then there must be a $0$-block in $A$ between $B_a$ and $B_b$. Thus suppose all 0-blocks between $B_a$ and $B_b$ are of length $1$, and that the index of each one appears in $A$. In Algorithm 1, when the indicator entry is $k=1$, the $0$-block is of length $1$ and the index of the $0$-block is added to $A$, the indicator coordinate stays $1$ and an index corresponding to a $1$-block is not added. This is a contradiction since then $b$ could not be added to $A$.

Now, let $S=V(T(B))\backslash J$ be a zero forcing set of $T(B)$ for some $J$ which includes a vertex of $B_i$ for each $i \in A$; we will show that $A\in \mathcal{A}$. The only case in which Algorithm 1 does not add every possible continuation of $(A,k)$ to $\mathcal{A}_s$ is when $k=1$ and $|B_{s-1}|=1$. The indicator $k=1$ implies that the after the most recently excluded $1$-vertex, every $0$-vertex has been excluded. Thus, in this case any vertex set excluding one vertex from each block with index in $A\cup\{s-1,s\}$ would exclude two $1$-vertices in a row without a $0$-vertex in between, contradicting Condition 2) of Theorem \ref{theorem: zfs characterization}. Thus no set of block indices with a corresponding zero forcing set contains $A\cup\{s-1,s\}$, so only the sets not in $\mathcal{A}$ do not have a corresponding zero forcing set.

The vertices in a block of $B$ are indistinguishable, so once we choose which blocks will be missing one vertex by choosing $A$, we only need to count how many ways there are to exclude one vertex from each chosen block, which can be done in exactly $\prod_{i\in A} |B_i|$ ways. This gives the first equality of \eqref{equation: threshold polynomial theorem}. The second equality follows immediately from combining terms with the same power of $x$.
\qed

\medskip
One result that follows from Theorem \ref{threshold zero forcing polynomial} is that there are arbitrarily large sets of nonisomorphic graphs that share a zero forcing polynomial; see Theorem \ref{thm_structural_zf} for details. Note also that the intermediate sets $\mathcal{A}_s$ produced by Algorithm 1 can be used together with (\ref{equation: threshold polynomial theorem}), at no additional cost, to give the zero forcing polynomials of the threshold graphs generated by $B_1\ldots B_s$ for any $s\leq t$ such that $B_s$ is a 1-block.

\section{Structural properties}
\label{section_structural}
In this section, we will show several structural results about the zero forcing polynomial of a graph. We will first show that the zero forcing polynomial of a disconnected graph is multiplicative over its connected components. 

\begin{proposition}
\label{prop_multiplicative}
If $G$ is a graph such that $G\simeq G_1\dot\cup G_2$, then $\mathcal{Z}(G;x)=\mathcal{Z}(G_1;x)\mathcal{Z}(G_2;x)$.
\end{proposition}
\begin{proof}
A zero forcing set of size $i$ in $G$ consists of a zero forcing set of size $i_1$ in $G_1$ and a zero forcing set of size $i_2=i-i_1$ in $G_2$. Since zero forcing sets of size $i_1$ and $i_2$ can be chosen independently in $G_1$ and $G_2$ for each $i_1\geq Z(G_1)$, $i_2\geq Z(G_2)$, and since $z(G_1;i_1)z(G_2;i_2)=0$ for each $i_1<Z(G_1)$ or $i_2<Z(G_2)$, it follows that $z(G;i)=\sum_{i_1+i_2=i}z(G_1;i_1)z(G_2;i_2)$. The left-hand-side of this equation is the coefficient of $x^i$ in $\mathcal{Z}(G)$, and since $\mathcal{Z}(G_1;x)=\sum_{i=Z(G_1)}^{|V(G_1)|}z(G_1;i)x^i$ and $\mathcal{Z}(G_2;x)=\sum_{i=Z(G_2)}^{|V(G_2)|}z(G_2;i)x^i$, the right-hand-side of the equation is the coefficient of $x^i$ in $\mathcal{Z}(G_1;x)\mathcal{Z}(G_2;x)$. Thus, $\mathcal{Z}(G_1;x)\mathcal{Z}(G_2;x)$ and $\mathcal{Z}(G;x)$ have the same coefficients and the same degree, so they are identical. 
\end{proof}

The next result lists some other basic facts about the zero forcing polynomial.

\begin{theorem}
\label{thm_structural_zf}
Let $G=(V,E)$ be a graph. Then,
\begin{enumerate}
\item $z(G;i)=0$ if and only if $i<Z(G)$,
\item Zero is a root of $\mathcal{Z}(G;x)$ of multiplicity $Z(G)$,
\item $\mathcal{Z}(G;x)$ is strictly increasing in $[0,\infty)$.
\item There exist arbitrarily large sets of graphs which all have the same zero forcing polynomials. 
\item A connected graph and a disconnected graph can have the same zero forcing polynomial.
\item The zero forcing polynomial of a graph can have complex roots. 
\end{enumerate}
\end{theorem}

\proof
The numbers of the proofs below correspond to the numbers in the statement of the theorem. 
\begin{enumerate}
\item[1.] Follows from the definition of $z(G;i)$.
\item[2.] Follows from the fact that the first nonzero coefficient of $\mathcal{Z}(G;x)$ corresponds to $x^{Z(G)}$.
\item[3.] Follows from the fact that all coefficients of $\mathcal{Z}(G;x)$ are nonnegative, and hence the derivative of $\mathcal{Z}(G;x)$ is a polynomial with nonnegative coefficients.
\item[4.] 
Fix $k > 2$ and consider the class of threshold graphs whose binary string $B$ has $(k-1)$ 1-blocks of sizes $2$ through $k$ (in any order) and between each two consecutive 1-blocks, a 0-block of size 2. These graphs are distinct for each permutation of the lengths of the 1-blocks, so this class of graphs has $(k-1)!$ members for any $k > 2$. Each such binary string $B$ has exactly $2k-3$ blocks and all blocks in $B$ have size at least $2$. Thus, any choice of block indices $A\in 2^{[2k-3]}$ corresponds to some zero forcing set of $T(B)$, since condition 2) of Theorem \ref{theorem: zfs characterization} is satisfied for such a binary string $B$ as long as condition 1) is satisfied. Theorem \ref{threshold zero forcing polynomial} shows that the zero forcing polynomial of a threshold graph only depends on which choices of block indices correspond to zero forcing sets, and on the sizes of the blocks, which are the same for all the graphs in this class. Thus, all graphs in this class have the same zero forcing polynomial. 
\item[5.] We claim that for any $a,b\geq 2$, $K_a\dot\cup K_b$ and $K_{a,b}$ have the same zero forcing polynomial. Indeed, $\mathcal{Z}(K_a\dot\cup K_b;x)=\mathcal{Z}(K_a;x)\mathcal{Z}(K_b;x)=(x^a+ax^{a-1})(x^b+bx^{b-1})=abx^{a+b-2}+(a+b)x^{a+b-1}+x^{a+b}=\mathcal{Z}(K_{a,b};x)$, where the first equality follows from Proposition \ref{prop_multiplicative}, the second equality follows from Proposition \ref{prop_clique}, and the last equality follows from Proposition~\ref{prop_multipartite}.
\item[6.] By Proposition \ref{prop_path}, $\mathcal{Z}(P_4;x)=2x+6x^2+4x^3+x^4$. This polynomial has two real roots $x=0$ and $x\approx -0.46$, and two complex roots $x\approx -1.77 \pm 1.11i$. \qed
\end{enumerate}

The next result of this section concerns the unimodality of the zero forcing polynomial. We first recall a well-known theorem due to Hall \cite{hall}. A \emph{matching} of $G=(V,E)$ is a set $M\subseteq E$ such that no two edges in $M$ have a common endpoint. A matching $M$ \emph{saturates} a
vertex $v$, if $v$ is an endpoint of some edge in $M$.

\begin{theorem}[Hall's Theorem \cite{hall}]
Let $G$ be a bipartite graph with parts $X$ and $Y$. $G$ has a matching that saturates every vertex in $X$ if and only if for all $S\subseteq X$, $|S|\leq |N(S)|$.
\end{theorem}

\begin{theorem}
\label{thm_hall}
Let $G=(V,E)$ be a graph on $n$ vertices. Then, $z(G;i)\leq z(G;i+1)$ for $1\leq i < \frac{n}{2}$.
\end{theorem}
\proof
For every zero forcing set $R$ of size $i$ and every $v\in V\backslash R$, $R\cup \{v\}$ is a zero forcing set of cardinality $i+1$. We will now show that to each zero forcing set of size $i$, $1\leq i < \frac{n}{2}$, we can associate a distinct zero forcing set of size $i+1$. Let $H$ be a bipartite graph with parts $X$ and $Y$, where the vertices of $X$ are zero forcing sets of $G$ of size $i$, and the vertices of $Y$ are all subsets of $V$ of size $i+1$; a vertex $x\in X$ is adjacent to a vertex $y\in Y$ in $H$ whenever $x\subseteq y$. For each $x\in X$, there are $n-i$ vertices not in $x$; thus, $d(x;H)=n-i$. Since a set of size $i+1$ has $i+1$ subsets of size $i$, it follows that for each $y\in Y$, $d(y;H)\leq i+1$. Suppose for contradiction that there exists a set $S\subseteq X$ such that $|S|>|N(S)|$. Since each vertex in $S$ has $n-i$ neighbors and since $|S|>|N(S)|$, by the Pigeonhole Principle, some vertex $v\in N(S)$ must have more than $n-i$ neighbors. Thus, $i+1\geq d(v;H)>n-i$, whence it follows that $i\geq \frac{n}{2}$; this contradicts the assumption that $i<\frac{n}{2}$. Thus, for every $S\subseteq X$, $|S|\leq |N(S)|$. By Theorem \ref{thm_hall}, $H$ has a matching that saturates all vertices of $X$. Thus, there are at least as many zero forcing sets of size $i+1$ as there are of size $i$, for $1\leq i < \frac{n}{2}$.
\qed
\medskip

A \emph{fort} of a graph $G$, as defined in \cite{caleb_thesis}, is a non-empty set $F\subset V$ such that no vertex outside $F$ is adjacent to exactly one vertex in $F$. Let $\mathcal{F}(G)$ be the set of all forts of $G$. In \cite{brimkov_caleb}, it was shown that $Z(G)$ is equal to the optimum of the following integer program:

\begin{eqnarray*}
\min&& \sum_{v\in V}s_v\\
\text{s.t.:} &&\sum_{v\in F}s_v\geq 1  \qquad\forall F\in \mathcal{F}(G)\\
&&s_v\in \{0,1\} \qquad\forall v\in V
\end{eqnarray*}

We now give a way to bound the number of constraints in this model using the zero forcing polynomial.

\begin{proposition}
Let $G$ be a graph of order $n$. Then, $|\mathcal{F}(G)|\leq 2^n-\mathcal{Z}(G;1)$.
\end{proposition}
\proof
$\mathcal{Z}(G;1)=\sum_{i=1}^nz(G;i)$ equals the number of zero forcing sets of $G$, and hence also the number of complements of zero forcing sets of $G$. By \cite[Theorem 5.2]{caleb_thesis}, a zero forcing set must intersect every fort. Thus, the complement of a zero forcing set cannot be a fort, and so the number of complements of zero forcing sets of $G$ is at most the number of sets of $G$ which are not forts, i.e., $2^n-|\mathcal{F}(G)|$. Thus, $\mathcal{Z}(G;1)\leq 2^n-|\mathcal{F}(G)|$, and the result follows.
\qed

\medskip

Next, we again use forts to bound the coefficients of the zero forcing polynomials of certain graphs.

\begin{proposition}
\label{prop_forts_coeffs}
Let $G$ be a graph of order n which has a fort $F$ with $|F|\leq Z(G)+1$. Then, $z(G;i)\leq {n\choose i}-{n-i-1\choose i}$ for $1\leq i\leq n$.
\end{proposition}
\proof

For $i>n/2$, clearly $z(G;i)\leq {n\choose i}-{n-i-1\choose i}={n\choose i}$; thus, we will assume henceforth that $i\in \{Z(G),\ldots,n/2\}$. Let $f=|F|$ and let $z'(G;i)$ denote the number of non-forcing sets of $G$ of size $i$. There are $\binom{n-f}{i}$ subsets of $V(G)\backslash F$ of size $i$, each of which is a non-forcing set of $G$. Thus, $z'(G;i) \geq \binom{n-f}{i}$. Since $f\leq Z(G)+1$ and $i\geq Z(G)$, it follows that $i+1\geq f$ and hence $n-f\geq n-(i+1)$. Thus, $z'(G;i)\geq \binom{n-f}{i}\geq \binom{n-(i+1)}{i}$, and so
\begin{equation*}
z(G;i)\leq \binom{n}i-z'(G;i)\leq \binom{n}i-\binom{n-i-1}{i}.\tag*{\qed}
\end{equation*}
\medskip

We conclude this section by showing that the bound from Proposition \ref{prop_forts_coeffs} also holds for Hamiltonian graphs. A \emph{Hamiltonian path} is a path which visits every vertex of a graph exactly once.

\begin{proposition}
\label{prop_bound}
Let $G$ be a graph of order $n$ which has a Hamiltonian path. Then, $z(G;i)\leq {n\choose i}-{n-i-1\choose i}$ for $1\leq i\leq n$, and this bound is sharp.
\end{proposition}
\begin{proof}
Label the vertices of $G$ in increasing order along a Hamiltonian path as $v_1,\ldots,v_n$. Label the vertices of a path $P_n$ in order from one end-vertex to the other as $v_1',\ldots,v_n'$. Let $S'$ be a non-forcing set of $P_n$. We claim that $S=\{v_i:v_i'\in S'\}$ is a non-forcing set of $G$. It was shown in Proposition \ref{prop_path} that the non-forcing sets of $P_n$ are precisely those which do not contain two adjacent vertices, nor an end-vertex of $P_n$. Thus, every vertex $v_i'\in S'$ has two uncolored neighbors, $v_{i-1}'$ and $v_{i+1}'$. Then, by construction, every vertex in $S$ has at least two uncolored neighbors, namely $v_{i-1}$ and $v_{i+1}$. Thus, no vertex in $S$ will be able to perform a force, so $S$ is not a zero forcing set of $G$. If $z'(H;i)$ denotes the number of non-forcing sets of size $i$ of a graph $H$, it follows that $z'(G;i)\geq z'(P_n;i)$. Thus, $z(G;i)={n\choose i}-z'(G;i)\leq {n\choose i}-z'(P_n;i)=z(P_n;i)={n\choose i}-{n-i-1\choose i}$. By Proposition \ref{prop_path}, the bound holds with equality for the path $P_n$.
\end{proof}

\subsection{Recognizing graphs by their zero forcing polynomials}

In this section, we identify several families of graphs which can be recognized by their zero forcing polynomials. As shown in Theorem \ref{thm_structural_zf}, it does not hold in general that if $\mathcal{Z}(G;x)=\mathcal{Z}(H;x)$, then $G\simeq H$. For example, complete multipartite graphs are generally not recognizable by their zero forcing polynomials. Wheels are also not recognizable by their zero forcing polynomials: for example, let $G$ be the graph obtained by subdividing an edge of $K_4$; then, $\mathcal{Z}(G;x)=8x^3+5x^4+x^5=\mathcal{Z}(W_5;x)$. On the other hand, it is easy to see from Theorem~\ref{thm_zf_polynomial_properties} that paths and complete graphs are recognizable by their zero forcing polynomials; we state this formally below.

\begin{proposition}
Let $G$ be a graph on $n\geq 1$ vertices. 
\begin{enumerate}
\item $\mathcal{Z}(G;x)=\mathcal{Z}(P_n;x)$ if and only if $G\simeq P_n$.
\item $\mathcal{Z}(G;x)=\mathcal{Z}(K_n;x)$ if and only if $G\simeq K_n$.
\end{enumerate}
\end{proposition}

We will now identify another nontrivial family of graphs which can be recognized by its zero forcing polynomials. Let $G=(V,E)\not\simeq P_n$ be a graph and $v$ be a vertex of degree at least $3$. A \emph{pendent path attached to} $v$ is a maximal set $P\subset V$ such that $G[P]$ is a connected component of $G-v$ which is a path, one of whose ends is adjacent to $v$ in $G$. The vertex $v$ will be called the \emph{base} of the path. A \emph{chord} is an edge joining two nonadjacent vertices in a cycle. A \emph{chorded cycle} is a cycle with added chords. Any cycle on $n$ vertices with a single chord will be denoted $C_n+e$. We will say two vertices in a chorded cycle are \emph{consecutive} if they are adjacent in the graph induced by the cycle minus the chords. Two pendent paths attached to a cycle are \emph{consecutive} if their bases in the cycle are consecutive. A graph $G=(V,E)$ is a \emph{graph of two parallel paths specified by $V_1$ and $V_2$} if $G\not\simeq P_n$, and if $V$ can be partitioned into nonempty sets $V_1$ and $V_2$ such that $P:=G[V_1]$ and $Q:=G[V_2]$ are paths, and such that $G$ can be drawn in the plane in such a way that $P$ and $Q$ are parallel line segments, and the edges between $P$ and $Q$ (drawn as straight line segments) do not cross. Note that if $G$ is a graph of two parallel paths, there may be several different partitions of $V$ into $V_1$ and $V_2$ which satisfy the conditions above. For example, let $G=(\{1,2,3,4,5\}$, $\{\{1,2\}$, $\{2,3\}$, $\{3,4\},\{4,5\},\{5,1\}\})$ be a cycle on 5 vertices. Then $G$ is a graph of two parallel paths that can be specified by $V_1=\{1\}$ and $V_2=\{2,3,4,5\}$, as well as by $V_1=\{1,2,3\}$ and $V_2=\{4,5\}$. Graphs of two parallel paths were introduced by Johnson et al. \cite{JLS09} in relation to graphs with maximum nullity 2. They were also used by Row \cite{row} in the following characterization.

\begin{theorem}[\cite{row}]
\label{zf2}
$Z(G)=2$ if and only if $G$ is a graph of two parallel paths.
\end{theorem}

\begin{theorem}
\label{thm: cycle polynomial distinct}
Let $G$ be a graph on $n\geq 3$ vertices and let $\phi(n,x)=\sum_{i=2}^n({n\choose i}-\frac{n}{i}{n-i-1\choose i-1})x^i$. $\mathcal{Z}(G;x)=\phi(n,x)$ if and only if $G$ is one of the following graphs: $C_n$, $C_n+e$, $P_2\dot\cup P_2$, $(P_4\dot\cup K_1)\lor K_1$.
\end{theorem}

\begin{proof}
By Proposition \ref{prop_cycle}, $\mathcal{Z}(C_n;x)=\phi(n,x)$. It suffices to show that $C_n$, $C_n + e$, $P_2\dot\cup P_2$ (for $n=4$), and $\mathscr{G}:=(P_4\dot\cup K_1)\lor K_1$ (for $n=6$) have the same zero forcing polynomial, and that this polynomial is distinct from the zero forcing polynomials of all other graphs $G$ with $Z(G)=2$. It can readily be verified that $\mathcal{Z}(P_2\dot\cup P_2;x)=\mathcal{Z}(C_4;x)$ and $\mathcal{Z}(\mathscr{G};x)=\mathcal{Z}(C_6;x)$. Now we will show that $\mathcal{Z}(C_n;x)=\mathcal{Z}(C_n+e;x)$.

\begin{claim} 
For $n\geq 4$, $\mathcal{Z}(C_n;x)=\mathcal{Z}(C_n+e;x)$. 
\end{claim}
\proof
Suppose $S$ is a zero forcing set of $C_n$ or $C_n+e$ which does not contain two consecutive vertices. Then, every vertex in $S$ has at least two uncolored neighbors (its two consecutive vertices), so no vertex in $S$ can force -- a contradiction. Thus, any zero forcing set of $C_n$ or $C_n+e$ contains two consecutive vertices.

Let $S$ be any zero forcing set of $C_n$. In $C_n+e$, each of the two consecutive vertices in $S$ can initiate a forcing chain (possibly of length zero) that terminates at an endpoint of the chord $e$. Then, when both endpoints of the chord have been colored, forcing can continue and color the rest of $C_n+e$. Thus, $S$ is also a zero forcing set of $C_n+e$. Similarly, if $S$ is a zero forcing set of $C_n+e$, $S$ contains two consecutive vertices, and thus $S$ is also a zero forcing set of $C_n$. Since the zero forcing sets of $C_n$ and $C_n+e$ are identical, it follows that $\mathcal{Z}(C_n;x)=\mathcal{Z}(C_n+e;x)$.
\qed
\medskip

It remains to show that $\mathcal{Z}(C_n;x)$ is distinct from the zero forcing polynomials of all other graphs with zero forcing number 2. By Theorem \ref{zf2}, graphs with zero forcing number 2 are graphs of two parallel paths. Let $G$ be a graph of two parallel paths with $n\geq 3$ vertices, different from $C_n$, $C_n+e$, $P_2\dot\cup P_2$, and $\mathscr{G}$. We will show that $\mathcal{Z}(G;x)$ and $\mathcal{Z}(C_n;x)$ differ in at least one coefficient. First we make a claim which follows from Proposition \ref{prop_cycle}.

\begin{claim}
\label{lem: cycle zfs}
For any $n$, $z(C_n;2)=n$. If $i > \frac{n}{2}$, every subset of $i$ vertices of $C_n$ is a zero forcing set. If $n$ is even, $C_n$ has exactly two non-forcing sets of size $\frac{n}{2}$.
\end{claim}

\begin{claim}
If there are no edges between the two parallel paths of $G$, then $\mathcal{Z}(G;x)\neq\mathcal{Z}(C_n;x)$. 
\end{claim}
\proof
$G$ consists of two disjoint paths, and there are at most four zero forcing sets of size $2$ in $G$, since any zero forcing set of size two must contain a vertex of degree at most 1 from each path. By Claim \ref{lem: cycle zfs}, if $n>4$, $\mathcal{Z}(C_n;x)\neq \mathcal{Z}(G;x)$. If $3\leq n\leq 4$, $G$ is either $P_2\dot\cup P_2$, $P_1\dot{\cup}P_3$, or $P_1\dot{\cup}P_2$. We have already noted that $\mathcal{Z}(C_4;x)=\mathcal{Z}(P_2\dot\cup P_2;x)$, and it can be readily verified that $\mathcal{Z}(C_4;x)\neq\mathcal{Z}(P_1\dot\cup P_3;x)$ and $\mathcal{Z}(C_3;x)\neq\mathcal{Z}(P_1\dot\cup P_2;x)$.
\qed
\begin{claim}
\label{claim single edge}
If there is a single edge between the two parallel paths of $G$, then $\mathcal{Z}(G;x)\neq\mathcal{Z}(C_n;x)$. 
\end{claim}
\proof
Let $P$ and $Q$ be the two parallel paths and $e=uv$ be the edge between them, with $u\in V(P)$ and $v\in V(Q)$. If $d(u)=d(v)=3$, then  $V(P)\cup \{v\}$ and $V(Q)\cup \{u\}$ are non-forcing sets, and one of them must contain more than $\frac{n}{2}$ vertices; by Claim \ref{lem: cycle zfs}, it follows that $\mathcal{Z}(G;x)\neq\mathcal{Z}(C_n;x)$. 

If one of $u$ and $v$, say $v$, does not have degree $3$, then since $G$ is not a single path, it must be that $d(u)=3$ and all other vertices have degree $1$ or $2$. Every zero forcing set of $G$ of size $2$ consists of either two degree $1$ vertices, or a degree $1$ vertex $y$ and a degree $2$ vertex adjacent to $u$ that is not between $u$ and $y$. This implies that there are $9-2\ell$ zero forcing sets of size $2$, where $\ell\in \{1,2,3\}$ is the number of neighbors of $u$ of degree $1$. By Claim \ref{lem: cycle zfs}, we can assume $n=9-2\ell$, since otherwise $z(C_n;2)\neq z(G;2)$. Then $\ell\neq 3$, since a graph on $3$ vertices cannot have a vertex of degree $3$. Figure \ref{figure: all worrysome spiders} shows all possible graphs with the desired properties. For each of these graphs, the filled-in vertices form a non-forcing set of size greater than $\frac{n}{2}$. Thus, by Claim \ref{lem: cycle zfs}, none of these graphs have the same zero forcing polynomial as $C_n$.
\qed

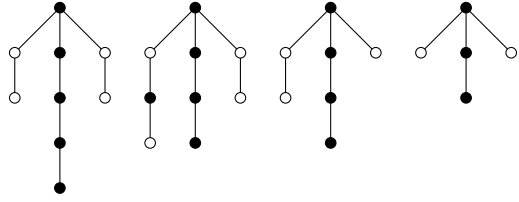
\begin{figure}[ht]
	\centering
	\begin{tikzpicture}[line cap=round,line join=round,>=triangle 45,x=0.6cm,y=0.6cm]
	
	\draw (-1,2)-- (-1,3);
	\draw (-1,3)-- (0,4);
	\draw (0,0)-- (0,1);
	\draw (0,1)-- (0,2);
	\draw (0,2)-- (0,3);
	\draw (0,3)-- (0,4);
	\draw (1,2)-- (1,3);
	\draw (1,3)-- (0,4);
	
	\draw (2,2)-- (2,3);
	\draw (2,3)-- (3,4);
	\draw (2,1)-- (2,2);
	\draw (3,1)-- (3,2);
	\draw (3,2)-- (3,3);
	\draw (3,3)-- (3,4);
	\draw (4,3)-- (3,4);
	\draw (4,2)-- (4,3);
	
	\draw (5,2)-- (5,3);
	\draw (5,3)-- (6,4);
	\draw (6,2)-- (6,3);
	\draw (6,3)-- (6,4);
	\draw (7,3)-- (6,4);
	\draw (6,1)-- (6,2);
	
	\draw (8,3)-- (9,4);
	\draw (9,2)-- (9,3);
	\draw (9,3)-- (9,4);
	\draw (10,3)-- (9,4);

	\draw [fill=black] (0,0) circle (2.0pt);
	\draw [fill=black] (0,1) circle (2.0pt);
	\draw [fill=black] (0,2) circle (2.0pt);
	\draw [fill=black] (0,3) circle (2.0pt);
	\draw [fill=black] (0,4) circle (2.0pt);
	\draw [fill=white] (-1,3) circle (2.0pt);
	\draw [fill=white] (-1,2) circle (2.0pt);
	\draw [fill=white] (1,3) circle (2.0pt);
	\draw [fill=white] (1,2) circle (2.0pt);
	
	\draw [fill=white] (2,1) circle (2.0pt);
	\draw [fill=black] (2,2) circle (2.0pt);
	\draw [fill=white] (2,3) circle (2.0pt);
	\draw [fill=black] (3,1) circle (2.0pt);
	\draw [fill=black] (3,2) circle (2.0pt);
	\draw [fill=black] (3,3) circle (2.0pt);
	\draw [fill=black] (3,4) circle (2.0pt);
	\draw [fill=white] (4,2) circle (2.0pt);
	\draw [fill=white] (4,3) circle (2.0pt);
	
	\draw [fill=white] (5,2) circle (2.0pt);
	\draw [fill=white] (5,3) circle (2.0pt);
	\draw [fill=black] (6,1) circle (2.0pt);
	\draw [fill=black] (6,2) circle (2.0pt);
	\draw [fill=black] (6,3) circle (2.0pt);
	\draw [fill=black] (6,4) circle (2.0pt);
	\draw [fill=white] (7,3) circle (2.0pt);
	
	\draw [fill=white] (8,3) circle (2.0pt);
	\draw [fill=black] (9,2) circle (2.0pt);
	\draw [fill=black] (9,3) circle (2.0pt);
	\draw [fill=black] (9,4) circle (2.0pt);
	\draw [fill=white] (10,3) circle (2.0pt);
	\end{tikzpicture}
	\caption{\label{figure: all worrysome spiders} All connected graphs with a single vertex $u$ of maximum degree 3, and $n=9-2\ell$ total vertices, where $\ell$ is the number of leaves adjacent to $u$. For each graph, non-forcing sets of size greater than $\frac{n}{2}$ are indicated by filled-in vertices.}
\end{figure}

\begin{claim}
\label{claim_spiders}
If $G$ has two pendent paths attached to the same vertex, then $\mathcal{Z}(G;x)\neq\mathcal{Z}(C_n;x)$. 
\end{claim}
\proof
By Claim \ref{claim single edge}, we can assume there are at least two edges between the parallel paths of $G$. Let $P$ and $Q$ be the two pendent paths which are attached to the same vertex $v$; then, $P\cup\{v\}\cup Q$ is one of the parallel paths of $G$, so $v$ is adjacent to at least two vertices of the other parallel path $G-(P\cup \{v\}\cup Q)$. Then, the sets $V(P)\cup V(Q)\cup\{v\}$ and $V(G)\setminus(V(P)\cup V(Q))$ are both non-forcing sets whose union is $V(G)$ and whose intersection is $\{v\}$. Thus, one of these sets must be of size greater than  $\frac{n}{2}$; by Claim \ref{lem: cycle zfs}, $\mathcal{Z}(G;x)\neq\mathcal{Z}(C_n;x)$. 
\qed
\medskip


We will assume henceforth that $G$ has at least two edges between the two parallel paths, and that any two pendent paths have distinct bases.
Thus, $G$ consists of a (possibly) chorded cycle $\widetilde{C}_m$ with up to four pendent paths; the vertices of $\widetilde{C}_m$ will be denoted $v_1,\dots,v_m$, where in $G[V(\widetilde{C}_m)]$, $v_i\sim v_{i+1}$ for $1\leq i\leq m-1$ and $v_1\sim v_m$. 
We will refer to a pendent path with base vertex $v_j\in V(\widetilde{C}_m)$ as $P_{v_j}$. 
We will denote by $\widetilde{C}_m(v_i,v_j)$ the section around the cycle $\widetilde{C}_m$ in clockwise orientation from $v_i$ to $v_j$, but not including $v_i$ and $v_j$, and $\widetilde{C}_m[v_i,v_j]=\widetilde{C}_m(v_i,v_j)\cup\{v_1,v_j\}$. Given two chords $e_1$ and $e_2$, we will say $v_i$ and $v_j$ are distinct endpoints of these chords if $v_i\neq v_j$ and $v_i,v_j\not\in e_1\cap e_2$. We now make two claims which follow easily from the definition of a zero forcing set. 

\begin{claim}
\label{lem: size 2 consecutives}
A zero forcing set of size 2 of a (possibly) chorded cycle with pendent paths consists either of two consecutive vertices, or of the degree 1 endpoints of consecutive pendent paths, or of the degree 1 endpoint of a pendent path and a vertex which is consecutive to the base of that pendent path.
\end{claim}

\begin{claim}
\label{lem: nonadjacent pendents}
If a (possibly) chorded cycle has two non-consecutive pendent paths, then no pair of consecutive vertices forms a zero forcing set. 
\end{claim}

\noindent Next we consider several cases based on the number of pendent paths in $G$. 

\medskip
\noindent
{\bf Case 1}. $G$ has no pendent paths.
\medskip

If $G$ has no pendent paths then $G$ is a chorded cycle, so by Claim \ref{lem: size 2 consecutives}, there are at most $n$ zero forcing sets of size $2$, namely the pairs of consecutive vertices along the cycle. Since $G$ is different from $C_n$ and $C_n+e$ and has a Hamilton cycle, we can assume $G$ is a graph of two parallel paths $P$ and $Q$ with at least four edges between the paths. This implies that one of the parallel paths, say $P$, contains at least two vertices of degree at least $3$. Let $u$ and $v$ be two vertices of degree at least $3$ in $P$ such that all vertices between $u$ and $v$ in $P$ have degree $2$. If $u\sim v$, then $\{u,v\}$ is not a zero forcing set since both $u$ and $v$ have two uncolored neighbors. Now suppose there is at least one vertex $x$ of degree $2$ with $x\sim u$ between $u$ and $v$ in $P$; see Figure \ref{figure: consecutives between two chords}. Then $\{x,u\}$ is not a zero forcing set since once $u$ forces along $P$ to $v$, no further forces can happen since both $u$ and $v$ will have at least two uncolored neighbors. In either case, $G$ has less than $n$ zero forcing sets of size $2$, so $z(G;2)<z(C_n;2)$.

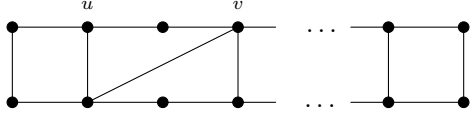
\begin{figure}[ht]
	\centering
	\begin{tikzpicture}[line cap=round,line join=round,>=triangle 45,x=1.0cm,y=1.0cm]
	\draw (1.,3.)-- (2.,3.);
	\draw (2.,3.)-- (3.,3.);
	\draw (1.,3.)-- (1.,2.);
	\draw (1.,2.)-- (2.,2.);
	\draw (2.,2.)-- (3.,2.);
	\draw (6.,3.)-- (6.,2.);
	\draw (3.,3.)-- (4.,3.);
	\draw (3.,2.)-- (4.,2.);
	\draw (2.,3.)-- (2.,2.);
	\draw (4.,3.)-- (4.,2.);
	\draw (4.,3.)-- (4.5,3.);
	\draw (4.,2.)-- (4.5,2.);
	\draw (6.,3.)-- (5.5,3.);
	\draw (6.,2.)-- (5.5,2.);
	\draw (4.78,3.1) node[anchor=north west] {$\dots$};
	\draw (4.76,2.1) node[anchor=north west] {$\dots$};
	\draw (2.,2.)-- (4.,3.);
	\draw (6.,3.)-- (7.,3.);
	\draw (7.,3.)-- (7.,2.);
	\draw (7.,2.)-- (6.,2.);
	\begin{scriptsize}
	\draw [fill=black] (1.,3.) circle (2.0pt);
	\draw [fill=black] (2.,3.) circle (2.0pt);
	\draw[color=black] (2,3.3) node {$u$};
	\draw [fill=black] (3.,3.) circle (2.0pt);
	\draw [fill=black] (1.,2.) circle (2.0pt);
	\draw [fill=black] (2.,2.) circle (2.0pt);
	\draw [fill=black] (3.,2.) circle (2.0pt);
	\draw [fill=black] (6.,3.) circle (2.0pt);
	\draw [fill=black] (6.,2.) circle (2.0pt);
	\draw [fill=black] (4.,3.) circle (2.0pt);
	\draw[color=black] (4,3.3) node {$v$};
	\draw [fill=black] (4.,2.) circle (2.0pt);
	\draw [fill=black] (7.,3.) circle (2.0pt);
	\draw [fill=black] (7.,2.) circle (2.0pt);
	\end{scriptsize}
	\end{tikzpicture}
	\caption{\label{figure: consecutives between two chords} Consecutive vertices between two chords.}
\end{figure}

\medskip
\noindent
\textbf{Case 2.} $G$ has one pendent path.
\medskip

Let $P_{v_1}=u_1,u_2, \dots , u_l$ be the pendent path. Then, $G= \widetilde{C}_m +P_{v_1}$ together with the edge $v_1u_l$, and $|V(G)|=m+l=n$. By Claim \ref{lem: cycle zfs}, it suffices to show that $G$ contains a set of vertices of size $i > \frac{n}{2}$ that is not a zero forcing set or at least three sets of size $\frac{n}2$ that are not forcing.

\medskip
\noindent
\textbf{Subcase 1.} Suppose either $l=1$ and $n$ is odd, or $l \geq 2$. 
\medskip

If $n$ is even, let $i= n/2 +1$. If $n$ is odd, let $i= (n+1)/2$. Choose a set of colored vertices of size $i$ as follows: color every vertex of $P_{v_1}$, color $v_1$, and then proceed around the cycle $\widetilde{C}_{m}$ starting from $v_1$, coloring every other vertex until $i$ vertices have been colored. Then, the colored vertices in $\widetilde{C}_m$ are all mutually nonconsecutive, and every colored vertex in $G$ has either zero or at least two uncolored neighbors. Thus, we have found a non-forcing set of size $i>\frac{n}{2}$ and $z(G;i)\neq z(C_n;i)$.

\medskip
\noindent
\textbf{Subcase 2.}  Suppose $l=1$ and $n$ is even.
\medskip

If $n\geq 8$, it follows that $m\geq 7$. Consider the sets consisting of $v_1$, $u_1$, and $\frac{m-3}2$ non-consecutive vertices from $\{v_3,\dots,v_{m-1}\}$; there are at least 3 ways to choose $\frac{m-3}2$ non-consecutive vertices for $m\geq 7$, so we have found 3 non-forcing sets of size $2+\frac{m-3}2=\frac{n}2$, as desired. Now, assume $n=6$; Figure \ref{figure: all chorded cycles on five vertices} shows all possible graphs with the desired properties. One of these graphs is $\mathscr{G}$; we have already noted that $\mathcal{Z}(C_6;x)=\mathcal{Z}(\mathscr{G};x)$. If $G$ is any of the other six graphs in Figure \ref{figure: all chorded cycles on five vertices}, it can be readily verified $\mathcal{Z}(C_6;x)\neq\mathcal{Z}(G;x)$. Finally, if $n=4$, then $G$ is a triangle with a pendent vertex, and it can be verified that $z(G;2)>z(C_4;2)$.

\begin{figure}[ht]
	\centering
	\begin{tikzpicture}[line cap=round,line join=round,>=triangle 45,x=0.69cm,y=0.69cm]
	\draw (0,0)-- (.5,.5);
	\draw (.5,.5)-- (1,1);
	\draw (1,1)-- (2,1);
	\draw (2,1)-- (2,0);
	\draw (2,0)-- (1,0);
	\draw (1,0)-- (.5,.5);
	\draw (1,1)-- (2,0);
	\draw (1,1)-- (1,0);
	
	\draw (3,0)-- (3.5,.5);
	\draw (3.5,.5)-- (4,1);
	\draw (4,1)-- (5,1);
	\draw (5,1)-- (5,0);
	\draw (5,0)-- (4,0);
	\draw (4,0)-- (3.5,.5);
	\draw (5,1)-- (3.5,.5);
	\draw (5,1)-- (4,0);
	
	\draw (6,0)-- (6.5,.5);
	\draw (6.5,.5)-- (7,1);
	\draw (7,1)-- (8,1);
	\draw (8,1)-- (8,0);
	\draw (8,0)-- (7,0);
	\draw (7,0)-- (6.5,.5);
	\draw (6.5,.5)-- (8,1);
	\draw (6.5,.5)-- (8,0);
	
	\draw (-3,0)-- (-2.5,.5);
	\draw (-2.5,.5)-- (-2,1);
	\draw (-2,1)-- (-1,1);
	\draw (-1,1)-- (-1,0);
	\draw (-1,0)-- (-2,0);
	\draw (-2,0)-- (-2.5,.5);
	\draw (-1,1)-- (-2.5,.5);
	
	\draw (-6,0)-- (-5.5,.5);
	\draw (-5.5,.5)-- (-5,1);
	\draw (-5,1)-- (-4,1);
	\draw (-4,1)-- (-4,0);
	\draw (-4,0)-- (-5,0);
	\draw (-5,0)-- (-5.5,.5);
	\draw (-4,0)-- (-5,1);
	
	\draw (-9,0)-- (-8.5,.5);
	\draw (-8.5,.5)-- (-8,1);
	\draw (-8,1)-- (-7,1);
	\draw (-7,1)-- (-7,0);
	\draw (-7,0)-- (-8,0);
	\draw (-8,0)-- (-8.5,.5);
	\draw (-8,0)-- (-8,1);
	
	\draw (-12,0)-- (-11.5,.5);
	\draw (-11.5,.5)-- (-11,1);
	\draw (-11,1)-- (-10,1);
	\draw (-10,1)-- (-10,0);
	\draw (-10,0)-- (-11,0);
	\draw (-11,0)-- (-11.5,.5);

	\draw [fill=black] (0,0) circle (2.0pt);
	\draw [fill=black] (1,0) circle (2.0pt);
	\draw [fill=black] (2,0) circle (2.0pt);
	\draw [fill=black] (3,0) circle (2.0pt);
	\draw [fill=black] (4,0) circle (2.0pt);
	\draw [fill=black] (5,0) circle (2.0pt);
	\draw [fill=black] (6,0) circle (2.0pt);
	\draw [fill=black] (7,0) circle (2.0pt);
	\draw [fill=black] (8,0) circle (2.0pt);
	\draw [fill=black] (-1,0) circle (2.0pt);
	\draw [fill=black] (-2,0) circle (2.0pt);
	\draw [fill=black] (-3,0) circle (2.0pt);
	\draw [fill=black] (-4,0) circle (2.0pt);
	\draw [fill=black] (-5,0) circle (2.0pt);
	\draw [fill=black] (-6,0) circle (2.0pt);
	\draw [fill=black] (-7,0) circle (2.0pt);
	\draw [fill=black] (-8,0) circle (2.0pt);
	\draw [fill=black] (-10,0) circle (2.0pt);
	\draw [fill=black] (-9,0) circle (2.0pt);
	\draw [fill=black] (-11,0) circle (2.0pt);
	\draw [fill=black] (-12,0) circle (2.0pt);
	\draw [fill=black] (-1,1) circle (2.0pt);
	\draw [fill=black] (-2,1) circle (2.0pt);
	\draw [fill=black] (-4,1) circle (2.0pt);
	\draw [fill=black] (-5,1) circle (2.0pt);
	\draw [fill=black] (-7,1) circle (2.0pt);
	\draw [fill=black] (-8,1) circle (2.0pt);
	\draw [fill=black] (-10,1) circle (2.0pt);
	\draw [fill=black] (-11,1) circle (2.0pt);
	\draw [fill=black] (.5,.5) circle (2.0pt);
	\draw [fill=black] (-2.5,.5) circle (2.0pt);
	\draw [fill=black] (-5.5,.5) circle (2.0pt);
	\draw [fill=black] (-8.5,.5) circle (2.0pt);
	\draw [fill=black] (-11.5,.5) circle (2.0pt);
	\draw [fill=black] (1,1) circle (2.0pt);
	\draw [fill=black] (2,1) circle (2.0pt);
	\draw [fill=black] (3.5,.5) circle (2.0pt);
	\draw [fill=black] (4,1) circle (2.0pt);
	\draw [fill=black] (5,1) circle (2.0pt);
	\draw [fill=black] (6.5,.5) circle (2.0pt);
	\draw [fill=black] (7,1) circle (2.0pt);
	\draw [fill=black] (8,1) circle (2.0pt);
	
	\draw[color=black] (9,0.5) node {$\leftarrow\mathscr{G}$};
	\end{tikzpicture}
	\caption{\label{figure: all chorded cycles on five vertices} All graphs of two parallel paths on six vertices with a single pendent path, of length 1.}
\end{figure}
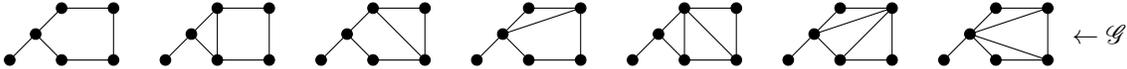

\medskip
\noindent
\textbf{Case 3.} $G$ has two pendent paths.
\medskip

If the two pendent paths are not consecutive, by Claim \ref{lem: nonadjacent pendents}, no pair of consecutive vertices can force the graph. Then, by Claim \ref{lem: size 2 consecutives}, the only possible zero forcing sets of size two contain the degree 1 endpoint of a pendent path $P_{v_i}$ and a vertex which is consecutive to its base $v_i$, or two degree 1 endpoints of consecutive pendent paths. Since $|\widetilde{C}_m| \geq 3$ and $G$ has two pendent paths, $n \geq 5$. Then, because the two pendent paths are not consecutive, there are at most $2 \cdot 2=4<n$ zero forcing sets of size 2, so  $z(G;2)<z(C_4;2)$.

Now assume the two pendent paths of $G$ are consecutive. let $P_{v_1}=u_1,u_2, \dots, u_{l_1}$ and $P_{v_2}=w_1,w_2, \dots, w_{l_2}$ be the pendent paths, with $u_{l_1}$ adjacent to $v_1$ and $w_{l_2}$ adjacent to $v_2$; see Figure \ref{figure: 2 consecutive pendents}. Without loss of generality, suppose $l_1 \geq l_2$.

\begin{figure}[ht]
	\centering
	\begin{tikzpicture}[line cap=round,line join=round,>=triangle 45,x=1.0cm,y=1.0cm]
	\draw (2.,5.)-- (2.5,5.);
	\draw (3.5,5.)-- (4.,5.);
	\draw (2.,4.)-- (2.5,4.);
	\draw (3.5,4.)-- (4.,4.);
	\draw (2.76,5.22) node[anchor=north west] {$\dots$};
	\draw (2.78,4.16) node[anchor=north west] {$\dots$};
	\draw (4.,5.)-- (5.,5.);
	\draw (4.,4.)-- (5.,4.);
	\draw (5.,5.)-- (5.,4.);
	\draw [shift={(5.75862068966,4.28793103448)}] plot[domain=0.628918118774:2.38783691482,variable=\t]({1.*1.04045545913*cos(\t r)+0.*1.04045545913*sin(\t r)},{0.*1.04045545913*cos(\t r)+1.*1.04045545913*sin(\t r)});
	\draw [shift={(5.77297297297,4.48243243243)}] plot[domain=3.69956313707:5.85005244369,variable=\t]({1.*0.911168628087*cos(\t r)+0.*0.911168628087*sin(\t r)},{0.*0.911168628087*cos(\t r)+1.*0.911168628087*sin(\t r)});
	\draw (6.5,4.94) node[anchor=north west] {$\vdots$};
	\draw (2.72,5.76) node[anchor=north west] {$P_{v_2}$};
	\draw (2.9,3.56) node[anchor=north west] {$P_{v_1}$};
	\draw (5.42,6.24) node[anchor=north west] {$\widetilde{C}_m$};
	\begin{scriptsize}
	\draw [fill=black] (2.,5.) circle (2.0pt);
	\draw[color=black] (2.12,5.2) node {$w_{1}$};
	\draw [fill=black] (4.,5.) circle (2.0pt);
	\draw[color=black] (4.14,5.2) node {$w_{l_2}$};
	\draw [fill=black] (2.,4.) circle (2.0pt);
	\draw[color=black] (2.02,3.7) node {$u_{1}$};
	\draw [fill=black] (4.,4.) circle (2.0pt);
	\draw[color=black] (4.16,3.7) node {$u_{l_1}$};
	\draw [fill=black] (5.,5.) circle (2.0pt);
	\draw[color=black] (5.18,5.38) node {$v_2$};
	\draw [fill=black] (5.,4.) circle (2.0pt);
	\draw[color=black] (5.18,3.6) node {$v_1$};
	\draw [fill=black] (6.,5.3) circle (2.0pt);
	\draw[color=black] (6.26,5.5) node {$v_{3}$};
	\draw [fill=black] (6.,3.6) circle (2.0pt);
	\draw[color=black] (6.26,3.2) node {$v_{m}$};
	\end{scriptsize}
	\end{tikzpicture}
	\caption{\label{figure: 2 consecutive pendents}$G$ has two pendent paths which are consecutive. }
\end{figure}
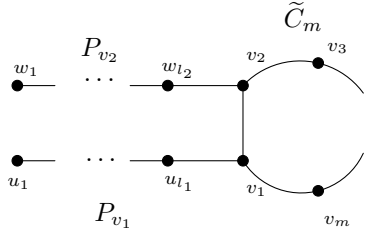

\medskip
\noindent
\textbf{Subcase 1.} Suppose $l_1 \geq 2$. 
\medskip

If $l_2 \geq 2$, let $S=V(P_{v_1}) \cup \{v_1,v_3, \dots, v_{p}\} \cup \{w_2\}$ where $p=m-1$ if $m$ is even and $p=m-2$ if $m$ is odd. It is easy to verify that $S$ is not a zero forcing set, and that $|S| \geq l_1 + \frac{m-1}{2} + 1 \geq \frac{l_1+l_2+m-1+2}{2} = \frac{n+1}{2}$. Since there is a non-forcing set of size greater than $\frac{n}{2}$, by  Claim \ref{lem: cycle zfs}, $\mathcal{Z}(G;x)\neq\mathcal{Z}(C_n;x)$. The case $l_1 > l_2 = 1$ is handled similarly. 

\medskip
\noindent
\textbf{Subcase 2.} Suppose $l_1=l_2=1$. 
\medskip

If $\widetilde{C}_m$ has no chords, then any pair of consecutive vertices on $\widetilde{C}_m$, except $\{v_1,v_2\}$, are a zero forcing set. In addition, $\{u_1, v_m\}$, $\{w_1, v_3\}$, $\{w_1,v_1\}$, and $\{u_1,v_2\}$ are all zero forcing sets. Thus, there are at least $n-2-1+4=n+1$ zero forcing sets of size 2, so $z(G;2)>z(C_n;2)$. Thus, we may assume $\widetilde{C}_m$ contains chords. Suppose there exists a chord $v_iv_j$, $i <j$, without $v_1$ or $v_2$ as an endpoint. By Claim \ref{lem: size 2 consecutives}, there are at most $m+5=n+3$ zero forcing sets of $G$ of size 2: $m$ pairs of consecutive vertices from $\widetilde{C}_m$, and the sets $\{u_1, v_m\}$, $\{w_1, v_3\}$, $\{w_1,v_1\}$, $\{u_1,v_2\}$, and  $\{u_1,w_1\}$. However, $\{v_1,v_2\}$ is not a forcing set, no consecutive pair of vertices in $\widetilde{C}_m[v_j,v_i]$ is a forcing set, and because of the chord, $\{w_1,v_3\}$ and $\{u_1,v_m\}$ are not forcing sets. It follows that $G$ has at most $n-2$ zero forcing sets of size 2, so $z(G;2)<z(C_n;2)$.

Now suppose every chord has $v_1$ or $v_2$ as an endpoint. Without loss of generality, let $v_1v_i$ be a chord such that $i$ is minimum. Clearly $i\geq 3$; suppose first that $i=3$. If $m$ is odd, the set $\{u_1,v_1\}\cup\{v_2,v_4,\ldots,v_{m-1}\}$ is a non-forcing set of size $\frac{n+1}{2}$. If $m$ is even, $\{u_1,v_1\}\cup\{v_2,v_4,\ldots,v_{m-2}\}$, $\{u_1,v_1,v_3,\ldots,v_{m-1}\}$, and $\{w_1,v_2,v_4,\ldots,v_m\}$ are non-forcing sets of size $\frac{n}{2}$. In either case, by Claim \ref{lem: cycle zfs}, $\mathcal{Z}(G;x)\neq\mathcal{Z}(C_n;x)$.
Now suppose that $i\geq 4$. It is easy to see that no pair of consecutive vertices $\{v_k,v_{k+1}\}$ in $\widetilde{C}_m[v_2,v_i]$ is a forcing set, and that $\{v_1,v_2\}$ and $\{v_1,w_1\}$ are not  forcing sets. Since $i\geq 4$, there are at least 2 pairs of consecutive vertices in $\widetilde{C}_m[v_2,v_i]$. Hence, together with $\{v_1,v_2\}$, there are at least 3 pairs of consecutive vertices in $\widetilde{C}_m$ that are not forcing sets; moreover, $\{v_1,w_1\}$ is a set consisting of a degree 1 endpoint of a pendent path and a vertex which is consecutive to the base of that pendent path. By Claim \ref{lem: size 2 consecutives}, $G$ has at most $m-3+(2\cdot 2-1)+1=m+1=n-1$ forcing sets of size 2, so $z(G;2)<z(C_n;2)$.

\medskip
\noindent
\textbf{Case 4.} $G$ has three pendent paths.
\medskip

First we will consider the case when the three pendent paths are all mutually consecutive. Let the paths be $P_{v_1}=u_1, \dots, u_{l_1}, P_{v_2}=w_1, \dots, w_{l_2}, P_{v_3}=x_1, \dots, x_{l_3}$ and $V(\widetilde{C}_m) = \{v_1,v_2,v_3\}$, where $v_i$ is the base of $P_{v_i}$, $1\leq i\leq 3$; see Figure \ref{figure: 3 consecutive pendents}. It is easy to see that $G$ has nine zero forcing sets of size 2: any pair of degree 1 vertices, and any degree 1 vertex in pendent path $P_{v_i}$ together with a neighbor of $v_i$ in $\widetilde{C}_m$. Since $C_n$ has $n$ zero forcing sets of size 2, it follows that if $n\neq 9$, $z(G;2)\neq z(C_n;2)$. Suppose henceforth that $n=9$; this implies that $l_1+l_2+l_3=6$. By Claim \ref{lem: cycle zfs}, any set of five vertices of $C_9$ is a zero forcing set. It can then be verified (similarly to Claim \ref{claim_spiders}) that each of the three nonisomorphic graphs resulting from the different possibilities for the values of $l_1,l_2,l_3$ contains a non-forcing set of size 5. Thus, $z(G;5)<z(C_9;5)$.

\begin{figure}[ht]
	\centering
	\begin{tikzpicture}[line cap=round,line join=round,>=triangle 45,x=1.0cm,y=1.0cm]
	\draw (5.,4.)-- (4.,3.);
	\draw (4.,3.)-- (6.,3.);
	\draw (6.,3.)-- (5.,4.);
	\draw (5.,4.)-- (5.,4.5);
	\draw (5.,6.)-- (5.,6.5);
	\draw (4.,3.)-- (3.5,3.);
	\draw (1.5,3.)-- (2.,3.);
	\draw (6.,3.)-- (6.5,3.);
	\draw (8.5,3.)-- (8.,3.);
	\draw (4.88,6.08) node[anchor=north west] {$\vdots$};
	\draw (2.2,3.1) node[anchor=north west] {$\dots$};
	\draw (7.2,3.1) node[anchor=north west] {$\dots$};
	\draw (7.06,3.88) node[anchor=north west] {$P_{v_3}$};
	\draw (2.34,3.8) node[anchor=north west] {$P_{v_1}$};
	\draw (5.44,5.7) node[anchor=north west] {$P_{v_2}$};
	\draw (4.78,3.76) node[anchor=north west] {$\widetilde{C}_m$};
	\draw (3.5,3.)-- (3.,3.);
	\draw (6.5,3.)-- (7.,3.);
	\draw (5.,4.5)-- (5.,5.);
	\begin{scriptsize}
	\draw [fill=black] (5.,4.) circle (2.0pt);
	\draw[color=black] (5.36,4.1) node {$v_2$};
	\draw [fill=black] (4.,3.) circle (2.0pt);
	\draw[color=black] (4.18,3.38) node {$v_1$};
	\draw [fill=black] (6.,3.) circle (2.0pt);
	\draw[color=black] (6.18,3.38) node {$v_3$};
	\draw [fill=black] (5.,4.5) circle (2.0pt);
	\draw[color=black] (5.34,4.7) node {$w_{l_2}$};
	\draw [fill=black] (5.,6.5) circle (2.0pt);
	\draw[color=black] (5.34,6.68) node {$w_1$};
	\draw [fill=black] (3.5,3.) circle (2.0pt);
	\draw[color=black] (3.7,3.38) node {$u_{l_1}$};
	\draw [fill=black] (1.5,3.) circle (2.0pt);
	\draw[color=black] (1.68,3.38) node {$u_1$};
	\draw [fill=black] (6.5,3.) circle (2.0pt);
	\draw[color=black] (6.7,3.38) node {$x_{l_3}$};
	\draw [fill=black] (8.5,3.) circle (2.0pt);
	\draw[color=black] (8.68,3.38) node {$x_1$};
	\end{scriptsize}
	\end{tikzpicture}
	\caption{\label{figure: 3 consecutive pendents} $G$ has three pendent paths which are mutually consecutive.}
\end{figure}
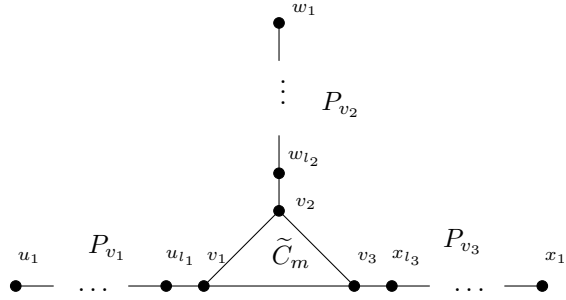

Now, we will assume that the three pendent paths are not mutually consecutive; for this to be possible, it must be that $m\geq 4$ and thus $n\geq 7$. Let $v'_1,v'_2,v'_3$ be vertices in $\{v_1,\ldots,v_m\}$ such that $P_{v'_1}$ and $P_{v'_3}$ are pendent paths with bases $v'_1$ and $v'_3$ that are not consecutive, and the base $v'_2$ of the third pendent path $P_{v'_2}$ lies in $\widetilde{C}_m(v'_1,v'_3)$. By Claim \ref{lem: nonadjacent pendents}, no pair of consecutive vertices in $\widetilde{C}_m$ can force $G$. Moreover, the degree one endpoint of $P_{v'_2}$ together with a consecutive vertex of $v'_2$ cannot force  $G$. Thus, there are at most four zero forcing sets of size 2 consisting of the degree 1 endpoint of $P_{v'_1}$ or $P_{v'_3}$ and a consecutive vertex of $v'_1$ or $v'_3$; there are also at most two zero forcing sets of size two consisting of the degree 1 endpoints of pendent paths that are consecutive. Thus, $G$ has at most $6<n$ zero forcing sets of size two, so $z(G;2)<z(C_n;2)$.

\medskip
\noindent
\textbf{Case 5.} $G$ has four pendent paths. 
\medskip

We have that $m \geq 4$ and $n\geq 8$ since $G$ has four pendent paths. Since no pendent paths are attached to the same vertex, at least one pair of the pendent paths must be non-adjacent; thus, by Claim \ref{lem: nonadjacent pendents}, no pair of consecutive vertices in $\widetilde{C}_m$ can force $G$. It is easy to see that no degree one endpoint of a pendent path $P$ together with a vertex consecutive to its base in $\widetilde{C}_m$ forms a zero forcing set of $G$. Thus, $G$ has at most $4<n$ zero forcing sets of size two: those consisting of degree 1 endpoints of pendent paths that are consecutive along the cycle. Thus, $z(G;2)<z(C_n;2)$.

\vspace{0.25cm}

We have shown that in all cases, $\mathcal{Z}(G;x)$ and $\mathcal{Z}(C_n;x)$ differ in at least one coefficient. This completes the proof of Theorem \ref{thm: cycle polynomial distinct}.
\end{proof}

\section{Conclusion}
\label{section_conclusion}

In this paper, we studied the enumeration problem associated with zero forcing by introducing the zero forcing polynomial of a graph. We characterized the extremal coefficients of $\mathcal{Z}(G;x)$, presented closed form expressions for the zero forcing polynomials of several families of graphs, and explored various structural properties of $\mathcal{Z}(G;x)$. 

We offer two conjectures regarding the coefficients of the zero forcing polynomial; evidence for these conjectures was shown in Theorem \ref{thm_hall} and Propositions \ref{prop_forts_coeffs} and \ref{prop_bound}, respectively (the conjectures also hold for all families of graphs whose zero forcing polynomials were characterized in Section \ref{section_characterizations}).

\begin{conjecture}
For any graph $G$, $\mathcal{Z}(G;x)$ is unimodal.
\end{conjecture}

\begin{conjecture}
For any graph $G$ on $n$ vertices, $z(G;i)\leq z(P_n;i)$, for $1\leq i\leq n$.
\end{conjecture}

Another direction for future work is to derive conditions which guarantee that a polynomial $P$ is or is not the zero forcing polynomial of some graph. In particular, it would be interesting to find other families of graphs which can be recognized by their zero forcing polynomials. It would also be interesting to characterize all zero forcing sets (or at least all minimum zero forcing sets) of some other nontrivial families of graphs such as trees and grids. For example, given a grid $P_m\Box P_n$, $m \leq n$, with the usual plane embedding and corresponding vertex coordinates, any path of order $m$ with one end-vertex at a corner of the grid and monotone coordinates of all other vertices is a minimum zero forcing set. A similar characterization of all other minimum zero forcing sets of grids could be pursued.

A graph polynomial $f(G;x)$ satisfies a \emph{linear recurrence relation} if $f(G;x)=\sum_{i=1}^k g_i(x)f(G_i;x)$, where the $G_i$'s are obtained from $G$ using vertex or edge elimination operations, and the $g_i$'s are fixed rational functions. For example, the chromatic polynomial $P(G;x)$ satisfies the deletion-contraction recurrence $P(G;x)=P(G-e;x)-P(G/e;x)$. Similarly, a \emph{splitting formula} for a graph polynomial $f(G;x)$ is an expression for $f(G;x)$ in terms of the polynomials of certain subgraphs of $G$; one such formula was shown in Proposition \ref{prop_multiplicative}. In view of this, it would be interesting to investigate the following question:

\begin{question}
Are there linear recurrence relations for $\mathcal{Z}(G;x)$, or splitting formulas for $\mathcal{Z}(G;x)$ based on cut vertices or separating sets?
\end{question}
Answering these questions would be useful for computational approaches to the problem; in particular, a linear recurrence relation would allow the zero forcing polynomial of a graph to be computed recursively.

\section*{Acknowledgments}
We thank Leslie Hogben, Ken Duna, and Adam Purcilly for valuable discussions. This research was partially funded by NSF-DMS Grants 1604458, 1604773, 1604697 and 1603823, and by Simons Foundation Collaboration Grant 316262.

\bibliographystyle{abbrv}
\bibliography{zf_bib}

\begin{thebibliography}{10}

\bibitem{AIM-Workshop}
{AIM~Special~Work~Group}.
\newblock Zero forcing sets and the minimum rank of graphs.
\newblock {\em Linear Algebra and its Applications}, 428(7):1628--1648, 2008.

\bibitem{dom_poly4}
S.~Akbari, S.~Alikhani, and Y.-H. Peng.
\newblock Characterization of graphs using domination polynomials.
\newblock {\em European Journal of Combinatorics}, 31(7):1714--1724, 2010.

\bibitem{edge_cover_poly}
S.~Akbari and M.~R. Oboudi.
\newblock On the edge cover polynomial of a graph.
\newblock {\em European Journal of Combinatorics}, 34(2):297--321, 2013.

\bibitem{dom_poly1}
S.~Alikhani and Y.-H. Peng.
\newblock Introduction to domination polynomial of a graph.
\newblock {\em Ars Combinatoria}, 114:257--266, 2014.

\bibitem{Barioli}
F.~Barioli, W.~Barrett, S.~M. Fallat, H.~T. Hall, L.~Hogben, B.~Shader, P.~Van
  Den~Driessche, and H.~Van Der~Holst.
\newblock Zero forcing parameters and minimum rank problems.
\newblock {\em Linear Algebra and its Applications}, 433(2):401--411, 2010.

\bibitem{zf_tw}
F.~Barioli, W.~Barrett, S.~M. Fallat, H.~T. Hall, L.~Hogben, B.~Shader,
  P.~van~den Driessche, and H.~Van Der~Holst.
\newblock Parameters related to tree-width, zero forcing, and maximum nullity
  of a graph.
\newblock {\em Journal of Graph Theory}, 72(2):146--177, 2013.

\bibitem{proptime_oriented}
A.~Berliner, C.~Bozeman, S.~Butler, M.~Catral, L.~Hogben, B.~Kroschel, J.~C.-H.
  Lin, N.~Warnberg, and M.~Young.
\newblock Zero forcing propagation time on oriented graphs.
\newblock {\em Discrete Applied Mathematics}, 2017.

\bibitem{birkhoff}
G.~D. Birkhoff.
\newblock A determinant formula for the number of ways of coloring a map.
\newblock {\em The Annals of Mathematics}, 14(1/4):42--46, 1912.

\bibitem{bondy}
J.~A. Bondy and U.~S.~R. Murty.
\newblock {\em Graph Theory with Applications}, volume 290.
\newblock Macmillan London, 1976.

\bibitem{brimkov_thesis}
B.~Brimkov.
\newblock {\em Graph Coloring, Zero Forcing, and Related Problems}.
\newblock PhD thesis, 2017.

\bibitem{brimkov_caleb}
B.~Brimkov, C.~C. Fast, and I.~V. Hicks.
\newblock Computational approaches for zero forcing and related problems.
\newblock {\em arXiv:1704.02065}, 2017.

\bibitem{brimkov2}
B.~Brimkov and I.~V. Hicks.
\newblock Complexity and computation of connected zero forcing.
\newblock {\em Discrete Applied Mathematics}, 229:31--45, 2017.

\bibitem{dom_poly2}
J.~I. Brown and J.~Tufts.
\newblock On the roots of domination polynomials.
\newblock {\em Graphs and Combinatorics}, 30(3):527--547, 2014.

\bibitem{brylawski}
T.~H. Brylawski.
\newblock A decomposition for combinatorial geometries.
\newblock {\em Transactions of the American Mathematical Society},
  171:235--282, 1972.

\bibitem{zf_physics}
D.~Burgarth, D.~D'Alessandro, L.~Hogben, S.~Severini, and M.~Young.
\newblock Zero forcing, linear and quantum controllability for systems evolving
  on networks.
\newblock {\em IEEE Transactions on Automatic Control}, 58(9):2349--2354, 2013.

\bibitem{quantum1}
D.~Burgarth and V.~Giovannetti.
\newblock Full control by locally induced relaxation.
\newblock {\em Physical Review Letters}, 99(10):100501, 2007.

\bibitem{logic1}
D.~Burgarth, V.~Giovannetti, L.~Hogben, S.~Severini, and M.~Young.
\newblock Logic circuits from zero forcing.
\newblock {\em Natural Computing}, 14(3):485--490, 2015.

\bibitem{butler}
S.~Butler, J.~Grout, and H.~T. Hall.
\newblock Using variants of zero forcing to bound the inertia set of a graph.
\newblock {\em Electronic Journal of Linear Algebra}, 30(1):1, 2015.

\bibitem{throttling}
S.~Butler and M.~Young.
\newblock Throttling zero forcing propagation speed on graphs.
\newblock {\em Australasian Journal of Combinatorics}, 57:65--71, 2013.

\bibitem{throttling2}
J.~Carlson, L.~Hogben, J.~Kritschgau, K.~Lorenzen, M.~S. Ross, S.~Selken, and
  V.~V. Martinez.
\newblock Throttling positive semidefinite zero forcing propagation time on
  graphs.
\newblock 2017.
\newblock available at:
  \url{https://orion.math.iastate.edu/lhogben/research/CHKLRSV2017-PSDthrottle.pdf}.

\bibitem{kforcing3}
Y.~Caro and R.~Pepper.
\newblock Dynamic approach to k-forcing.
\newblock {\em Theory and Applications of Graphs}, 2(2), 2015.

\bibitem{reliability_poly}
M.~Chari and C.~J. Colbourn.
\newblock Reliability polynomials: a survey.
\newblock {\em Journal of Combinatorics, Information and System Sciences},
  22(3-4):177--193, 1997.

\bibitem{chvatal}
V.~Chv{\'a}tal and P.~L. Hammer.
\newblock Aggregation of inequalities in integer programming.
\newblock {\em Annals of Discrete Mathematics}, 1:145--162, 1977.

\bibitem{connected_dom_poly}
G.~Deepak and N.~Soner.
\newblock Connected domination polynomial of a graph.
\newblock {\em International Journal of Mathematical Archive}, 4(11):90--96,
  2014.

\bibitem{vertex_cover_poly}
F.~M. Dong, M.~D. Hendy, K.~L. Teo, and C.~H. Little.
\newblock The vertex-cover polynomial of a graph.
\newblock {\em Discrete Mathematics}, 250(1-3):71--78, 2002.

\bibitem{disease}
P.~A. Dreyer and F.~S. Roberts.
\newblock Irreversible k-threshold processes: Graph-theoretical threshold
  models of the spread of disease and of opinion.
\newblock {\em Discrete Applied Mathematics}, 157(7):1615--1627, 2009.

\bibitem{positive_zf2}
J.~Ekstrand, C.~Erickson, H.~T. Hall, D.~Hay, L.~Hogben, R.~Johnson,
  N.~Kingsley, S.~Osborne, T.~Peters, J.~Roat, et~al.
\newblock Positive semidefinite zero forcing.
\newblock {\em Linear Algebra and its Applications}, 439(7):1862--1874, 2013.

\bibitem{tuttemain}
J.~A. Ellis-Monaghan and C.~Merino.
\newblock Graph polynomials and their applications {I}: {T}he {T}utte
  polynomial.
\newblock In {\em Structural Analysis of Complex Networks}, pages 219--255.
  Springer, 2011.

\bibitem{tuttemain2}
J.~A. Ellis-Monaghan and C.~Merino.
\newblock Graph polynomials and their applications ii: Interrelations and
  interpretations.
\newblock In {\em Structural Analysis of Complex Networks}, pages 257--292.
  Springer, 2011.

\bibitem{caleb_thesis}
C.~C. Fast.
\newblock {\em Novel Techniques for the Zero-Forcing and p-Median Graph
  Location Problems}.
\newblock PhD thesis, 2017.

\bibitem{signed_zf}
F.~Goldberg and A.~Berman.
\newblock Zero forcing for sign patterns.
\newblock {\em Linear Algebra and its Applications}, 447:56--67, 2014.

\bibitem{clique_poly}
H.~Hajiabolhassan and M.~Mehrabadi.
\newblock On clique polynomials.
\newblock {\em Australasian Journal of Combinatorics}, 18:313--316, 1998.

\bibitem{hall}
P.~Hall.
\newblock On representatives of subsets.
\newblock {\em Journal of the London Mathematical Society}, 10(1):26--30, 1935.

\bibitem{powerdom3}
T.~W. Haynes, S.~M. Hedetniemi, S.~T. Hedetniemi, and M.~A. Henning.
\newblock Domination in graphs applied to electric power networks.
\newblock {\em SIAM Journal on Discrete Mathematics}, 15(4):519--529, 2002.

\bibitem{indep_poly}
C.~Hoede and X.~Li.
\newblock Clique polynomials and independent set polynomials of graphs.
\newblock {\em Discrete Mathematics}, 125(1-3):219--228, 1994.

\bibitem{proptime1}
L.~Hogben, N.~Kingsley, S.~Meyer, S.~Walker, and M.~Young.
\newblock Propagation time for zero forcing on a graph.
\newblock {\em Discrete Applied Mathematics}, 160(13):1994--2005, 2012.

\bibitem{fractional_zf}
L.~Hogben, K.~F. Palmowski, D.~E. Roberson, and M.~Young.
\newblock Fractional zero forcing via three-color forcing games.
\newblock {\em Discrete Applied Mathematics}, 213:114--129, 2016.

\bibitem{dom_poly5}
S.~Jahari and S.~Alikhani.
\newblock Domination polynomials of k-tree related graphs.
\newblock {\em International Journal of Combinatorics}, 2014, 2014.

\bibitem{JLS09}
C.~R. Johnson, R.~Loewy, and P.~A. Smith.
\newblock The graphs for which the maximum multiplicity of an eigenvalue is
  two.
\newblock {\em Linear and Multilinear Algebra}, 57(7):713--736, 2009.

\bibitem{jones}
V.~F. Jones.
\newblock A polynomial invariant for knots via von neumann algebras.
\newblock {\em Bulletin of the American Mathematical Society}, 12:103--111,
  1985.

\bibitem{dom_poly3}
T.~Kotek, J.~Preen, F.~Simon, P.~Tittmann, and M.~Trinks.
\newblock Recurrence relations and splitting formulas for the domination
  polynomial.
\newblock {\em Electronic Journal of Combinatorics}, 19(3), 2012.

\bibitem{row}
D.~D. Row.
\newblock A technique for computing the zero forcing number of a graph with a
  cut-vertex.
\newblock {\em Linear Algebra and its Applications}, 436(12):4423--4432, 2012.

\bibitem{stanley}
R.~P. Stanley.
\newblock Acyclic orientations of graphs.
\newblock In {\em Classic Papers in Combinatorics}, pages 453--460. Springer,
  1987.

\bibitem{tutte2}
W.~T. Tutte.
\newblock On dichromatic polynomials.
\newblock {\em Journal of Combinatorial Theory}, 2(3):301--320, 1967.

\bibitem{proptime2}
N.~Warnberg.
\newblock Positive semidefinite propagation time.
\newblock {\em Discrete Applied Mathematics}, 198:274--290, 2016.

\bibitem{whitney}
H.~Whitney.
\newblock A logical expansion in mathematics.
\newblock {\em Bulletin of the American Mathematical Society}, 38(8):572--579,
  1932.

\bibitem{fast_mixed_search}
B.~Yang.
\newblock Fast--mixed searching and related problems on graphs.
\newblock {\em Theoretical Computer Science}, 507:100--113, 2013.

\bibitem{powerdom2}
M.~Zhao, L.~Kang, and G.~J. Chang.
\newblock Power domination in graphs.
\newblock {\em Discrete Mathematics}, 306(15):1812--1816, 2006.

\end{thebibliography}
\end{document}